\documentclass[11pt]{article}
\usepackage[english]{babel}

\usepackage{latexsym}
\usepackage{amsmath}
\usepackage{amsfonts}
\usepackage{mathrsfs}
\usepackage{amstext}
\usepackage{amssymb}
\usepackage{amsthm}       
\usepackage{hyperref}

\usepackage{todonotes}
\usepackage{soul}

\usepackage{xspace}
\usepackage{xcolor}

\newcommand\R{\ensuremath{\mathbb{R}}}

\newcommand\N{\ensuremath{\mathbb{N}}}
\newcommand\eps{\ensuremath{\varepsilon}}

\allowdisplaybreaks

\newtheorem{theorem}{Theorem}[section]
\newtheorem{prop}[theorem]{Proposition}
\newtheorem{cor}[theorem]{Corollary}

\newtheorem{defn}[theorem]{Definition}
\newtheorem{rem}[theorem]{Remark}
\newtheorem{lem}[theorem]{Lemma}

\usepackage[margin=1.2in]{geometry}

\usepackage{graphicx, fullpage}	

\begin{document}

\title{Asymptotic behavior of  generalized capacities 
with applications to eigenvalue perturbations: the higher dimensional case}
\author{\ }

\author{
Laura Abatangelo\thanks{Dipartimento di Matematica, Politecnico di Milano, P.zza Leonardo da Vinci 32, 20133 Milano, Italy \texttt{laura.abatangelo@polimi.it}} ,\ 
  Corentin L\'ena\thanks{Dipartimento di Tecnica e Gestione dei Sistemi Industriali (DTG), Universit\`a degli Studi di Padova,  Stradella S. Nicola 3, 36100 Vicenza, Italy \texttt{corentin.lena@unipd.it}} \thanks{Dipartimento di Matematica ``Tullio Levi-Civita'', Universit\`a degli Studi di Padova, via Trieste 63, 35121 Padova, Italy},
and Paolo Musolino\thanks{Dipartimento di Scienze Molecolari e Nanosistemi, Universit\`a Ca' Foscari Venezia, via Torino 155, 30172 Venezia Mestre, Italy \texttt{paolo.musolino@unive.it}}}

\date{September 10, 2023}


\maketitle 

\noindent
{\bf Abstract.} We provide a full series expansion of a generalization of the so-called $u$-capacity related to the Dirichlet-Laplacian in dimension three and higher, extending the results of \cite{AbBoLeMu21, AbLeMu22} dealing with the planar case. We apply the result in order to study the asymptotic behavior {of perturbed eigenvalues}
when Dirichlet conditions are imposed on a small regular subset of the domain of the eigenvalue problem. 

\vspace{11pt}


\vspace{11pt}


\section{Introduction}

This paper is devoted to the analysis of the asymptotic behavior of a generalization of the condenser capacity of a bounded domain of $\mathbb{R}^d$ ($d\geq 3$) with a small hole of size $\varepsilon>0$, as the parameter $\varepsilon$ approaches the degenerate value $0$. Such a generalization, as we shall see, has  proved to be extremely useful for the analysis of the behavior of simple and multiple eigenvalues for the Dirichlet-Laplacian in a bounded domain with a small hole. A careful investigation of the behavior of such a generalized capacity has been  carried out in \cite{AbBoLeMu21, AbLeMu22} in the planar case, with the corresponding applications to the study of the eigenvalues in perforated domains. Here, instead, we wish to consider the case of dimension $d$ greater than or equal to $3$. As we shall see, the cases of dimension $d=2$ and of $d\geq 3$ need to be treated separately because of the different aspect of the fundamental solution of the Laplace operator.

Before introducing the generalization of the capacity we are going to study, we begin by recalling the notion of {\it condenser capacity} and a first generalization of it, known as {\it $u$-capacity}. 
 
We recall that for a bounded, connected open set $\Omega$   of $\mathbb{R}^d$ and a  compact subset $K$ of $\Omega$, the {\it (condenser) capacity} of $K$ in $\Omega$ is 
\begin{equation} \label{eq:cap}
\mathrm{Cap}_{\Omega}(K)\equiv \mathrm{inf} \bigg\{\int_{\Omega} |\nabla f|^2\, dx \colon f \in H^1_0(\Omega)\ \mathrm{and}\ f-\eta_K \in H^1_0(\Omega \setminus K)\bigg\}\, ,
\end{equation}
where $\eta_K$ is a fixed smooth function such that $\mathrm{supp}\, \eta_K \subseteq \Omega$ and $\eta_K \equiv 1$ in a neighborhood of $K$. As is well known, the infimum in \eqref{eq:cap} is achieved by a function $V_K \in H^1_0(\Omega)$ such that $V_K-\eta_K \in H^1_0(\Omega\setminus K)$ so that
\[
\mathrm{Cap}_\Omega(K)=\int_{\Omega}|\nabla V_K|^2 \, dx\, ,
\]
where $V_K$ ({\it capacitary potential}) is the unique solution of the Dirichlet problem
\begin{equation}\label{eq:dircapK}
\left\{
\begin{array}{ll}
\Delta V_K=0&\text{ in }\Omega \setminus K\,,\\
V_K=0&\text{on } \partial\Omega\,,\\
V_K=1 &\text{ on }K\,.
\end{array}
\right.
\end{equation}
By saying that $V_K$ solves \eqref{eq:dircapK} we mean that $V_K \in H_0^1(\Omega)$, $V_K - \eta_K \in H_0^1(\Omega \setminus K)$, and
\[
\int_{\Omega \setminus K}\nabla V_K \cdot \nabla \phi\, dx=0 \qquad \forall\phi \in H_0^1(\Omega\setminus K).
\]
Moreover, if $\Omega$ and $K$ are sufficiently regular, one can reformulate the boundary conditions of problem \eqref{eq:dircapK} in the trace sense. Let us also point out that  $\mathrm{Cap}_\Omega(K)$ and $V_K$ do not depend on the choice of $\eta_K$, but only on the set $K$.

The following set of properties for the domains will play a key role in our analysis, so we summarize it in a definition.
\begin{defn} \label{def:domain} We say that $\Omega\subset\R^d$ is an \emph{admissible domain} if  
\begin{enumerate}
	\item $\Omega$ is open, bounded and connected;
	\item $\Omega$ is in the Schauder class $C^{1,\alpha}$ for some $\alpha\in]0,1[$;
	\item $0\in\Omega$ and $\R^d\setminus\overline{\Omega}$ is connected. 
	\end{enumerate}
\end{defn}

In the case where $d\ge3$, we can use the notion of \emph{Newtonian capacity}. We restrict ourselves to the case where $K=\overline{\omega}$, with $\omega$ an admissible domain. Then, there exists a unique function
 \[W_K:\R^d\setminus\omega\to \R\]  
 that is continuous, real-analytic in the open set $\R^d\setminus K$, and satisfies
\begin{equation}\label{eq:NewtoncapK}
\left\{
\begin{array}{ll}
\Delta W_K=0&\text{ in } \mathbb{R}^d \setminus K\,,\\
W_K=1&\text{on } \partial K\,,\\
W_K(x)\to0 &\text{ for }x\to\infty\,.
\end{array} 
\right.
\end{equation}
The last condition {in \eqref{eq:NewtoncapK}} is equivalent to the requirement that $W_K$ be harmonic at infinity. We then define the Newtonian capacity as
\begin{equation}\label{eq:Newcap}
\mathrm{Cap}_{{\mathbb{R}^d}}(K)\equiv\int_{\R^d\setminus K}|\nabla W_K|^2 \, dx\, ,
\end{equation}
 the integral being finite.

 In order to study the asymptotic behavior {as $\varepsilon\to 0$}  of eigenvalues of the Dirichlet-Laplacian in perforated domains of the type
\[
\Omega_\varepsilon= \Omega \setminus \varepsilon \overline{\omega}\, ,
\]
where $\Omega,\ \omega\subset\R^d$ are admissible domains in the sense of Definition \ref{def:domain}, a generalization of the notion of capacity --the so-called {\it $u$-capacity}--, has been introduced (see Abatangelo, Felli, Hillairet, and L\'ena \cite{AbFeHiLe19}). 
 
\begin{defn}\label{def:ucap} Given a function $u\in H^1_0(\Omega)$, the $u$-capacity of a compact set $K\subseteq \Omega$ is 
 \begin{equation} \label{eq:ucap}
\mathrm{Cap}_{\Omega}(K,u)\equiv \mathrm{inf} \bigg\{\int_{\Omega} |\nabla f|^2\, dx \colon f \in H^1_0(\Omega)\ \mathrm{and}\ f-u \in H^1_0(\Omega \setminus K)\bigg\}\,. 
\end{equation}
The infimum in \eqref{eq:ucap} is achieved by a unique function $V_{K,u}\in H_0^1(\Omega)$, so that
\[
\mathrm{Cap}_\Omega(K,u)=\int_{\Omega}|\nabla V_{K,u}|^2 \, dx\, .
\]
We call $V_{K,u}$ the \emph{potential} associated with $u$ and $K$.

Furthermore, $V_{K,u}$ is the unique weak solution of  the Dirichlet problem
\[
\left\{
\begin{array}{ll}
\Delta V_{K,u}=0&\text{ in }\Omega \setminus K\,,\\
V_{K,u}=0&\text{ on } \partial\Omega\,,\\
V_{K,u}=u &\text{ on }K\,,
\end{array}
\right.
\]
where, by \emph{weak solution}, we mean  that $V_{K,u}\in H^1_0(\Omega)$, $u-V_{K,u}\in H_0^1(\Omega\setminus K)$ and $\int_\Omega \nabla V_{K,u}\cdot\nabla \varphi\, dx=0$ for all $\varphi\in H_0^1(\Omega\setminus K)$. 
\end{defn}

We extend Definition \ref{def:ucap}  to $H^1(\Omega)$ functions, by setting, for any $u \in H^1(\Omega)$, \[\mathrm{Cap}_\Omega(K,u) \equiv \mathrm{Cap}_\Omega(K,\eta_K u)\, ,\] where $\eta_K$ is a fixed smooth function such that $\mathrm{supp}\, \eta_K \subseteq \Omega$ and $\eta_K \equiv 1$ in a neighborhood of $K$. Here again,  $\mathrm{Cap}_\Omega(K,u)$ and the associated potential $V_{K,u}$ do not depend on the  choice of $\eta_K$, but only on $K$ and $u$.


With this tool in hand, one can obtain an asymptotic formula for the behavior of $N$-th eigenvalue $\lambda_N(\Omega_\varepsilon)$ of the Dirichlet-Laplacian in $\Omega_\varepsilon$ as $\varepsilon \to 0$, under the assumption that the $N$-th eigenvalue $\lambda_N(\Omega)$ of the Dirichlet-Laplacian in the unperturbed set $\Omega$ is simple.

In order to be more precise, we recall that if $\tilde{\Omega}$ is a bounded open set in $\mathbb{R}^{d}$  the  eigenvalue problem 
\[
 \begin{cases}
  -\Delta u = \lambda u &\text{in }\tilde\Omega\, ,\\
  u=0 &\text{on }\partial \tilde\Omega\, 
 \end{cases}
\]
 admits a sequence of real eigenvalues tending to infinity
\[
0<\lambda_1({\tilde{\Omega}})\leq\lambda_2({\tilde{\Omega}})\leq \dots \leq \lambda_N({\tilde{\Omega}})\leq \dots \to +\infty
\]
where every eigenvalue is repeated as many times as its multiplicity and, in particular, the first one is  simple if ${\tilde{\Omega}}$ is connected.
{The dependence of the spectrum of the Laplace operator upon domain perturbations has long been investigated, with particular attention to the case of sets with small perforations. { Here we mention, for example, Abatangelo, Felli, Hillairet, and L\'ena \cite{AbFeHiLe19}, Ammari, Kang, and Lee \cite{AmKaLe09}, Besson \cite{Be85}, Chavel and Feldman \cite{ChFe88}, Colbois and Courtois \cite{CoCo91}, Courtois \cite{Co95}, Felli, Noris, and Ognibene \cite{FeNoOgCalcVar2021}, Lamberti and Perin \cite{LaPe10}, Lanza de Cristoforis \cite{La12}, Maz'ya, Movchan, and Nieves \cite{MaMoNi17}, Maz'ya, Nazarov, and Plamenevski\u\i  \ \cite{MaNaPl84}, Ozawa \cite{Oz81b}, Rauch and Taylor \cite{RaTa75},  Samarski\u\i \  \cite{Sa48}.} In particular, we recall that, by Courtois \cite[Proof of Theorem 1.2]{Co95} and by Abatangelo, Felli, Hillairet, and L\'ena \cite[Theorem 1.4]{AbFeHiLe19}, an asymptotic formula for the eigenvalues can be obtained in terms of $\mathrm{Cap}_\Omega({\varepsilon\overline{\omega}},u_N)$, where $u_N$ is a $L^2(\Omega)$-normalized eigenfunction corresponding to $\lambda_N(\Omega)$: if $\Omega$ and $\omega$ are admissible domains and $\lambda_N(\Omega)$ is simple, this reads as
\[
\lambda_N(\Omega \setminus(\varepsilon\overline{\omega}))=\lambda_N(\Omega)+\mathrm{Cap}_\Omega({\varepsilon\overline{\omega}},u_N)+o(\mathrm{Cap}_\Omega({\varepsilon\overline{\omega}},u_N)) \qquad \text{as $\varepsilon \to 0^+$}\, . 
\]

Unfortunately, the above mentioned result {a priori only} holds for simple eigenvalues. In order to study multiple eigenvalues, {one needs a further generalization of the capacity.}
The eigenspace being of dimension greater than one, one needs to take into account the interaction between different eigenfunctions corresponding to the same eigenvalue. 
\begin{defn}\label{def:uvcap} Given $u^a,u^b\in H_0^1(\Omega)$, the $(u^a,u^b)$-capacity of a compact set $K\subseteq \Omega$ is 
\begin{equation*}
	\mathrm{Cap}_\Omega(K,u^a,u^b)\equiv\int_\Omega\nabla V_{K,u^a}\cdot\nabla V_{K,u^b}\, dx.
\end{equation*}
\end{defn}
Here above the symbol $\cdot$ denotes the scalar product in $\mathbb{R}^d$.

Again, as we have done for the $u$-capacity, we extend Definition \ref{def:uvcap}  to $H^1(\Omega)$ functions, by setting, for any pair of functions $u^a,u^b \in H^1(\Omega)$, \[\mathrm{Cap}_\Omega(K,u^a,u^b) \equiv \mathrm{Cap}_\Omega(K,\eta_K u^a, \eta_K u^b)\, ,\] where $\eta_K$ is a fixed smooth function such that $\mathrm{supp}\, \eta_K \subseteq \Omega$ and $\eta_K \equiv 1$ in a neighborhood of $K$.

With the new object  $\mathrm{Cap}_\Omega(K,u^a,u^b)$ we can clearly recover the classical condenser capacity as well as the $u$-capacity.

\begin{rem}
If $u \in H^1(\Omega)$ and $K$ is a compact {subset} of $\Omega$, then
\begin{equation*}
	\mathrm{Cap}_\Omega(K,u,u)=	\mathrm{Cap}_\Omega(K,u)\, .
\end{equation*}
Also, 
\begin{equation*}
	\mathrm{Cap}_\Omega(K,1,1)=	\mathrm{Cap}_\Omega(K)\, .
\end{equation*}
\end{rem}

In \cite{AbLeMu22} it has been shown that by exploiting the definition of $(u^a,u^b)$-capacity, we can 
obtain 
{the asymptotic behavior}
of multiple eigenvalues. More precisely,  if one has an asymptotic expansion of $\mathrm{Cap}_\Omega(\varepsilon \overline{\omega},u^a,u^b)$ where $u^a$ and $u^b$ are eigenfunctions corresponding to the same  multiple eigenvalue in the unperturbed set $\Omega$, then one can deduce {the asymptotic behavior of} the corresponding eigenvalues in the perforated set $\Omega \setminus (\varepsilon\overline{\omega})$. Therefore, in \cite{AbLeMu22}, we have obtained a detailed and fully constructive representation of $\mathrm{Cap}_\Omega(\varepsilon \overline{\omega},u^a,u^b)$ as a convergent series in case of dimension $d=2$. Since our method is based on potential theory, the higher dimensional case differs from the two-dimensional one (mainly due to the different aspect of the fundamental solution for the Laplace equation). The goal of the present paper is to obtain a representation of $\mathrm{Cap}_\Omega(\varepsilon \overline{\omega},u^a,u^b)$ for $\varepsilon$ close to $0$  in terms of a convergent series and to apply it to study the asymptotic behavior of Dirichlet eigenvalues in perforated domains, in the case $d\geq 3$. Such results extend the work done in \cite{AbLeMu22} for the case of dimension $d=2$. Moreover, taking $u^a=u^b=1$, one can immediately deduce the expansion for the condenser capacity $\mathrm{Cap}_{\Omega}(\varepsilon \overline{\omega})$ as $\varepsilon \to 0$. We observe that several authors have considered the asymptotic behavior of condenser capacities in several geometrical situations (see Dubinin \cite{Du14}, Lanza de Cristoforis \cite{La02,La04, La08}, Maz'ya, Nazarov, and Plamenevskij \cite[\S 8.1]{MaNaPl00i}, So\u{\i}bel'man \cite{So81}). As an example, we mention that the asymptotic expansion of the capacity as the hole shrinks to a point can be deduced from the analysis of energy integrals in perforated domains by Maz'ya, Nazarov, and Plamenevskij \cite[\S 8.1]{MaNaPl00i}. For example, in dimension two, they prove that there exists $\delta >0$ such that
\begin{equation}\label{eq:MNPexp}
\mathrm{Cap}_{\Omega}(\varepsilon\overline{\omega})=-\frac{2\pi}{\log \varepsilon + 2 \pi \big(H_{(0,0)}+N\big)} +o(\varepsilon^{\delta})\, , 
\end{equation}
for $\varepsilon$ small and positive, where $e^{2\pi N}$ is the {\it logarithmic capacity} (or {\it outer conformal radius}) of $\omega$ and $H_{(0,0)}$ is the value at $x=0$ of the unique harmonic function $h$ in $\Omega$ such that $h(x)=-\log|x|/(2\pi)$  for all $x\in\partial\Omega$. For the definition of logarithmic capacity we refer, {\it e.g.}, to Landkof \cite[p.~168]{La72}, P\'olya and Szeg\"o \cite[p.~2]{PoSz51},  and Pommerenke \cite[p.~332]{Po75}. In other words, $N$ in equation \eqref{eq:MNPexp} is defined as $1/(2\pi)$ multiplied by the logarithm of the logarithmic capacity of $\omega$. Moreover, expansions for the capacity for the case of several small inclusions can be deduced from the corresponding expansion of the capacitary potential obtained in Maz'ya, Movchan, and Nieves~\cite[\S 3.2.2]{MaMoNi13}.

A different method to analyze the asymptotic behavior of functionals in perforated domains is the so-called {\it Functional Analytic Approach} proposed by Lanza de Cristoforis in  \cite{La02,La04}. The goal of this method is to represent functionals in singularly perturbed domains by means of real analytic maps and possibly singular but explicitly known functions of the singular perturbation parameter (see Dalla Riva, Lanza de Cristoforis, and Musolino \cite{DaLaMu21} for an introduction). As far as the condenser capacity in dimension two is concerned, by  Lanza de Cristoforis \cite{La02,La04}, we know that there exist {$\varepsilon_1>0$}
and a real analytic function $\mathcal{R}$ from $]-\varepsilon_1,\varepsilon_1[$ to $]0,+\infty[$ such that
\[
\mathrm{Cap}_{\Omega}(\varepsilon\overline{\omega})=-\frac{2\pi }{\log \varepsilon + \log \mathcal{R}[\varepsilon]} \qquad \forall \varepsilon \in ]0,\varepsilon_1[\, .
\]

Such a result implies that there exists a real analytic map $\tilde{\mathcal{R}}$ from a neighborhood of $(0,0)$ in $\mathbb{R}^2$ with values in  $\mathbb{R}$ such that
\[
\mathrm{Cap}_{\Omega}(\varepsilon\overline{\omega})=\tilde{\mathcal{R}}\Big[\varepsilon, \frac{1}{\log \varepsilon}\Big] \, ,
\]
for $\varepsilon$ positive and close to $0$. As a consequence, we have that
\[
\mathrm{Cap}_{\Omega}(\varepsilon\overline{\omega})=\sum_{(k,l) \in \mathbb{N}^2} \gamma_{(k,l)}\varepsilon^k \Big(\frac{1}{\log \varepsilon}\Big)^l \, ,
\]
for $\varepsilon$ positive and small enough, where the double power series $\sum_{(k,l) \in \mathbb{N}^2} \gamma_{(k,l)}x_1^k x_2^l$ converges for $(x_1,x_2)$ in a neighborhood of $(0,0)$. Moreover, in Lanza de Cristoforis \cite{La08}, the Functional Analytic Approach has been applied to the capacity $\mathrm{Cap}_{\Omega}(\varepsilon\overline{\omega})$ in $\mathbb{R}^d$ with $d \geq 2$ and has shown the existence of two real analytic maps $V_1$, $V_2$ from a neighborhood of $0$ to $\mathbb{R}$ such that
\begin{equation}\label{eq:capLanza}
\mathrm{Cap}_{\Omega}(\varepsilon\overline{\omega})=\frac{-1}{V_1[\varepsilon]+V_2[\varepsilon]\Upsilon_d[\varepsilon]}
\end{equation}
for $\varepsilon$ positive and small enough, with  $\Upsilon_d[\varepsilon]$ defined as
\[
{\Upsilon_d[\varepsilon] \equiv}
\left\{
\begin{array}{ll}
\frac{1}{2\pi} \log \varepsilon & \text{if $d=2$}\, ,\\
\frac{1}{(2-d)s_d}  \varepsilon^{2-d} & \text{if $d\geq 3$}\, ,\\
\end{array}
\right.
\]
where $s_d$ denotes the $(d-1)$-dimensional measure of the unit sphere in $\mathbb{R}^d$. In dimension $d\geq 3$ formula \eqref{eq:capLanza} implies the possibility of representing $\mathrm{Cap}_{\Omega}(\varepsilon\overline{\omega})$ as a convergent power series in $\varepsilon$, even though the explicit computation of the coefficients of the series is not available.

In \cite{AbBoLeMu21, AbLeMu22}, we have improved the above mentioned two-dimensional results in two directions: first we have considered a generalization of the classical capacity (the $u$-capacity), then we have explicitly computed
{some of the coefficients of the}
series representing the $u$-capacity. This result has then been applied to the asymptotic behavior of simple and multiple eigenvalues of the Laplacian in perforated {domains} in dimension $d=2$. The purpose of the present paper is to consider the behavior of (generalized) capacities in dimension  three and higher together with its 
application to the  asymptotic behavior of Dirichlet-Laplacian eigenvalues.

We wish to stress that to obtain an asymptotic expansion of $\mathrm{Cap}_\Omega(\varepsilon\overline{\omega},u^a,u^b)$, we follow the lines of the computations of  \cite{AbBoLeMu21, AbLeMu22}. Although the strategy is similar, some changes need to be taken into account due to the different dimension (here $d\geq 3$, whereas $d=2$ in  \cite{AbBoLeMu21, AbLeMu22}). Since they deserve some attention and care, we decided to provide these computations in the present paper.

We do believe that explicit results in dimension $d\geq 3$ will be useful both to people interested in the behavior of capacities and to those interested {in} the asymptotic analysis of eigenvalues of the Laplacian in perforated domains. To the best of our knowledge, the explicit and constructive series expansion we obtain is new even for  the classical capacity.

We assume that $\Omega$ and $\omega$ are {admissible domains in the sense of Definition \ref{def:domain}}. 
These regularity assumptions on $\Omega$ and $\omega$ are convenient for the Functional Analytic Approach we adopt, although they can be relaxed to Lipschitz regularity as it has been done in Costabel, Dalla Riva, Dauge, and Musolino \cite{CoDaDaMu17}

\begin{defn}\label{d:function} Let $\Omega$ be an admissible domain. We call  $u$ an \emph{admissible function} if 
\begin{enumerate}
	\item $u\in H^1(\Omega)$;
	\item $u$ is real-analytic in a neighborhood of $0$. 
\end{enumerate}
\end{defn}

Our first main result, Theorem \ref{capk}, deals with the value of  $\mathrm{Cap}_\Omega(\varepsilon \overline{\omega},u^a,u^b)$ 
for $\varepsilon$ close to $0$. We  reformulate Theorem \ref{capk} in Theorem \ref{t:0} below. For the proof, we refer to Section \ref{sec5}.

\begin{theorem}\label{t:0}
Let $\Omega$, $\omega$ be admissible domains and let $u^a,u^b$ be admissible functions. Then there exist  $\varepsilon^\#_\mathrm{c}$ positive and small enough and a sequence $\{c^\#_n\}_{n\in \mathbb{N}}$ of real numbers such that 
\[
\mathrm{Cap}_\Omega(\varepsilon \overline{\omega},u^a,u^b)=\varepsilon^{d-2}\sum_{n=0}^\infty c^\#_n \varepsilon^n
\]
for all $\varepsilon\in]0,\varepsilon^\#_\mathrm{c}[$. 
\end{theorem}

In addition, we are able to give a simple description of the first few terms in the above series expansion. For an admissible function $u$ which is not identically zero in a neighborhood of $0$, we denote by $\kappa(u){\in \mathbb{N}}$ 
the order of vanishing of $u$ at $0$, so that $D^\gamma u(0)=0$  for all $|\gamma| <\kappa(u)$ and $D^\beta u(0)\neq 0$ for some $\beta \in \mathbb{N}^d$ with $|\beta|=\kappa(u)$ {(we use the standard multi-index notation throughout the paper)}. We define the \emph{principal part} of $u$ by
 \begin{equation*}
   u_\#(x)\equiv\sum_{\beta\in\N^d,\\ |\beta|=\kappa(u)}
\frac1{\beta!}{D^{\beta}u(0)x^{\beta}},
 \end{equation*}
 so that $u_\#$ is a homogeneous polynomial of degree $\kappa(u)$. {Note that this includes the case when $\kappa(u)=0$, {\it i.e.,} $u(0)\neq 0$. Then $u_\#=u(0)$.} We finally denote by $U$ the unique function, continuous in $\R^d\setminus{\omega}$,  which solves the exterior boundary value problem
 \begin{equation}\label{eq:Uext}
\left\{
\begin{array}{ll}
\Delta U=0&\text{ in }\R^d \setminus \overline\omega\,,\\
U=u_\# &\text{ on }\partial\omega\,,\\
U(x)\to0&\text{ for }|x|\to\infty\,.
\end{array}
\right.
\end{equation}
We now assume that the admissible functions $u^a,{u^b}$ have finite orders of vanishing $\kappa(u^a)=\overline{k}_a$ and $\kappa(u^b)=\overline{k}_b$ at $0$. We write
\begin{equation*}
\mathfrak C\left(\omega,(u^a)_\#,\,(u^b)_\#\right)\equiv\int_{\R^d\setminus\overline\omega}\nabla U^a\cdot\nabla U^b\,dx+\int_{\omega}\nabla u^a_\#\cdot\nabla u^b_\#\,dx\, ,
\end{equation*}
where, for $l=a,b$,  the function $u^l_\#$ is the principal part of $u^l$  and  the function $U^l$ is the solution of problem \eqref{eq:Uext} with $u_\#$ replaced by $u^l_\#$.
Then, ${c^\#_n}=0$ for all $n<\overline{k}_a+\overline{k}_b$ and
\begin{equation*}
	{c^\#_{\overline{k}_a+\overline{k}_b}}=\mathfrak C\left(\omega,(u^a)_\#,\,(u^b)_\#\right).
\end{equation*}
When $u$ is a single admissible function, with finite order of vanishing $\kappa(u)=\overline{k}$ at $0$, we simplify notation further by writing 
\begin{equation*}
	\mathfrak C(\omega,u_\#)\equiv \mathfrak C(\omega,u_\#,u_\#),
\end{equation*}
and we note that 
\begin{equation}\label{eq:frakC}
\mathfrak C(\omega,u_\#)=\int_{\R^d\setminus\overline\omega}\left|\nabla U\right|^2\,dx+\int_{\omega}\left|\nabla u_\#\right|^2\,dx
\end{equation}
is strictly positive. {In particular, $\mathrm{Cap}_\Omega(\varepsilon \overline{\omega},u)$ is asymptotic to $\mathfrak C(\omega,u_\#)\varepsilon^{2\overline{k}+d-2}$.}

{We then apply} the above 
results to the asymptotic behavior of eigenvalues in perforated domains. Our main results on this 
problem are the following Theorems \ref{t:1} and \ref{t:2} (see Theorems \ref{thm:eig1}, \ref{thm:eig2}, and \ref{thm:orderEVs}
{for more detailed statements}). We begin with Theorem \ref{t:1} which is a reformulation of  Theorems \ref{thm:eig1} and \ref{thm:eig2} and which deals with the case when the $N$-th eigenvalue in the unperturbed set $\Omega$ is simple. For the proof, we refer to Subsection \ref{subsec:simple}.
\begin{theorem}\label{t:1} 
Let $\lambda_N(\Omega)$ be a simple eigenvalue of the Dirichlet-Laplacian in 
{an admissible domain} $\Omega$. Let $u_N$ be a $L^2(\Omega)$-normalized eigenfunction associated to $\lambda_N(\Omega)$. Then
\begin{equation*}
\lambda_N(\Omega\setminus (\varepsilon \overline{\omega}))=\lambda_N(\Omega)+(u_N(0))^2\mathrm{Cap}_{{\mathbb{R}^d}}(\overline{\omega})
\varepsilon^{d-2}+o\Big(\varepsilon^{d-2}\Big)\qquad \text{as $\varepsilon \to 0^+$}\, {.}
\end{equation*}
Moreover, let $\overline{k}$ be the order of vanishing {of $u_N$} at $0$. 
Then 
\begin{equation*}
\lambda_N(\Omega\setminus (\varepsilon \overline{\omega}))=\lambda_N(\Omega)+\mathfrak C\left(\omega,(u_N)_\#\right)\varepsilon^{2\overline{k}+d-2}+o\Big(\varepsilon^{2\overline{k}+d-2}\Big)\qquad \text{as $\varepsilon \to 0^+$}\, {.}
\end{equation*}
\end{theorem}

Theorem \ref{t:2} here below is instead concerned with the case when $\lambda_N(\Omega)$ is an eigenvalue of multiplicity $m>1$. {The case of multiple eigenvalues is considered in Theorem \ref{thm:orderEVs} and its proof can be found in Subsection \ref{subsec:multiple}.
}

{
\begin{theorem}\label{t:2}Let $\lambda_N(\Omega)$ be an eigenvalue of multiplicity $m>1$  and let $E(\lambda_N(\Omega))$ be the associated eigenspace. There exists an orthonormal basis of $E(\lambda_N(\Omega))$,
\[v_1,\dots,v_i,\dots,v_m,\]
such that, for $1\le i\le m$,
\begin{equation*}
	\lambda_{N-1+i}(\Omega\setminus (\varepsilon \overline{\omega}))=\lambda_N(\Omega)+
\mathrm{Cap}_\Omega(\varepsilon \overline{\omega},v_i)+o\left(\mathrm{Cap}_\Omega(\varepsilon \overline{\omega},v_i)\right)\mbox{ as }\eps\to0^+\, {.}
\end{equation*}
Moreover, we have
\begin{equation*}
	\lambda_{N-1+i}(\Omega\setminus (\varepsilon \overline{\omega}))=\lambda_N(\Omega)+\mu_i\varepsilon^{2\overline{k}_i+d-2}+o\left(\varepsilon^{2\overline{k}_i+d-2}\right)\mbox{ as }\eps\to0^+\, {,}
\end{equation*}
where $\overline{k}_i$ is the order of vanishing of $v_i$ at $0$ and 
\begin{equation*}
	\mu_j=\mathfrak C\left(\omega,(v_i)_\#\right)>0.
\end{equation*}
\end{theorem}

\begin{rem} In the previous theorem, the non-increasing finite sequence of integers $\left(\overline{k}_i\right)$ and the non-decreasing finite sequence of positive numbers $(\mu_i)$ are obviously independent of the choice of a basis $(v_i)$ satisfying the properties, by uniqueness of  the asymptotic expansion.
\end{rem}
}

The paper is organized as follows. In Section \ref{s:intcap} we rewrite 
the boundary value problem associated to the generalization of the capacity in terms of integral equations, exploiting classical potential theory. In Sections \ref{sec3} and \ref{sec4} we obtain series expansions for the solutions of the integral equations and for an auxiliary function.
In Section \ref{sec5}, we deduce our main result on the series expansion of the generalized capacity $\mathrm{Cap}_\Omega(\varepsilon\overline{\omega},u^a,u^b)$ {for $\varepsilon$ close to $0$.} { In Section  \ref{s:vanish}, we compute the principal term of the asymptotic expansion of $\mathrm{Cap}_\Omega(\varepsilon \overline{\omega},u^a,u^b)$ under vanishing assumption for $u^a$ and $u^b$.} In Section \ref{sec6} we outline the proofs of Theorem \ref{t:1} and Theorem \ref{t:2}.
Moreover, we supplement
the present paper with a blow-up analysis for the $u$-capacity in the Appendix. It shows more clearly that the first term of the series expansion provided in Section \ref{sec5} can be seen as a suitable capacity in the whole space $\R^d$.

\section{Integral equation formulation of capacitary potentials}\label{s:intcap}

\subsection{Preliminaries and classical notions of potential theory}\label{prel1}

In this paper we consider the dimension 
\[
d\in \mathbb{N}\setminus \{0,1,2\}
\] 
and we study the asymptotic behavior of $\mathrm{Cap}_\Omega(\varepsilon\overline{\omega},u^a,u^b)$ as $\varepsilon \to 0$. To do so, we assume some smoothness  on the sets and on the functions $u^a$ and $u^b$. 
We work in the frame of Schauder classes and thus we assume that both $\Omega$ and $\omega$ are admissible domains in the sense of Definition \ref{def:domain}. We can obviously find a common Schauder class $C^{1,\alpha}$ to which they belong, up to taking a smaller $\alpha$.
Since $\Omega$ is open and contains $0$, and since $\omega$ bounded, it is clear that there exists
  $\varepsilon_\#$ such that
\begin{equation}\label{eq:varepssharp}
\varepsilon_\#>0\ \mathrm{and\ }\ \varepsilon\overline{\omega} \subseteq \Omega\ \mathrm{for\ all}\ \varepsilon\in]-\varepsilon_\#,\varepsilon_\#[\,.
\end{equation}  
To define our perforated domain, we set
\[
\Omega_\varepsilon\equiv\Omega\setminus (\eps\overline\omega)
\qquad\quad\forall\varepsilon\in]-\varepsilon_\#,\varepsilon_\#[\,.
\]
Clearly, $\Omega_\varepsilon$ is an open bounded connected subset of $\mathbb{R}^{d}$ of class $C^{1,\alpha}$ for all $\varepsilon\in]-\varepsilon_\#,\varepsilon_\#[\setminus\{0\}$. The boundary
$\partial \Omega_\varepsilon$ of $\Omega_\varepsilon$ is the union of  $\partial \Omega$ and $\partial (\varepsilon \omega)=\varepsilon\partial\omega$, for all
$\varepsilon\in]-\varepsilon_\#,\varepsilon_\#[\setminus \{0\}$. For $\varepsilon=0$, we have $\Omega_0=\Omega\setminus\{0\}$.  We also need some regularity on the functions $u^a, u^b$: {we assume that $u^a,u^b$ are admissible in the sense of Definition \ref{d:function}.}

Our goal is to provide accurate and explicit expansions for $\mathrm{Cap}_\Omega(\varepsilon\overline{\omega},u^a,u^b)$ in terms of $\varepsilon$, with particular emphasis on the 
{influence} of the geometry and the data of the problem on such formulas.

Our strategy will be the same of  \cite{AbBoLeMu21, AbLeMu22}, where we adopted the Functional Analytic Approach of Lanza de Cristoforis \cite{La02, La04} for the analysis of singularly perturbed boundary value problems (see Dalla Riva, Lanza de Cristoforis, and Musolino \cite{DaLaMu21} for a detailed presentation). This approach permits to deduce a representation of the solution or related functionals as {real-analytic} maps, and thus as convergent power series. 

To analyze $\mathrm{Cap}_\Omega(\varepsilon\overline{\omega},u^a,u^b)$, we modify the techniques of \cite{AbBoLeMu21,AbLeMu22}, where we considered $\mathrm{Cap}_\Omega(\varepsilon\overline{\omega},u)$ and $\mathrm{Cap}_\Omega(\varepsilon\overline{\omega},u^a,u^b)$ but only in the planar case. Indeed, we emphasize that for considering $d\geq 3$ some modifications need to be done, as it is customary when using  Potential Theory.

By the analyticity of $u^a$ and $u^b$ {(see 
Definition \ref{d:function})} and analyticity results for the composition operator (see  B\"{o}hme and Tomi~\cite[p.~10]{BoTo73}, 
Henry~\cite[p.~29]{He82}, Valent~\cite[Thm.~5.2, p.~44]{Va88}), we deduce that, possibly shrinking $\varepsilon_\#$, there exists two real analytic maps $U^a_\#$, $U^b_\#$ from $]-\varepsilon_\#,\varepsilon_\#[$ to $C^{1,\alpha}(\partial \omega)$ such that
\[
u^a(\varepsilon t)= U^a_\#[\varepsilon](t)\, ,\qquad u^b(\varepsilon t)= U^b_\#[\varepsilon](t) \, , \qquad \forall t \in \partial \omega\, ,\forall \varepsilon \in ]-\varepsilon_\#,\varepsilon_\#[
\]
(see Deimling~\cite[\S 15]{De85} for the 
definition and properties of analytic maps). Then for all $\varepsilon\in]-\varepsilon_\#,\varepsilon_\#[ \setminus\{0\}$, we denote by $u^a_\varepsilon$ and $u^b_\varepsilon$  the unique solutions in $C^{1,\alpha}(\overline{\Omega_\varepsilon})$ of the  problems
\begin{equation}\label{eq:direpsa}
\left\{
\begin{array}{ll}
\Delta u^a_\varepsilon=0&\text{ in }\Omega_\varepsilon\,,\\
u^a_\varepsilon(x)=0&\text{ for all }x\in\partial\Omega\,,\\
u^a_\varepsilon(x)=U^a_\#[\varepsilon](x/\varepsilon)&\text{ for all }x\in\varepsilon\partial\omega\,
\end{array}
\right.
\end{equation}
and
\[
\left\{
\begin{array}{ll}
\Delta u^b_\varepsilon=0&\text{ in }\Omega_\varepsilon\,,\\
u^b_\varepsilon(x)=0&\text{ for all }x\in\partial\Omega\,,\\
u^b_\varepsilon(x)=U^b_\#[\varepsilon](x/\varepsilon)&\text{ for all }x\in\varepsilon\partial\omega\,,
\end{array}
\right.
\]
respectively. By the Divergence Theorem, we see that
\[
\begin{split}
\mathrm{Cap}_\Omega(\varepsilon\overline{\omega},u^a,u^b)&=\int_{\Omega_\varepsilon}\nabla u^a_\varepsilon \cdot \nabla u^b_\varepsilon \, dx + \int_{\varepsilon \omega}\nabla u^a \cdot \nabla u^b \, dx\\
&= -\varepsilon^{d-2}\int_{\partial \omega}\nu_{\omega}(t)\cdot \nabla \Big(u^a_\varepsilon(\varepsilon t)\Big) u^b(\varepsilon t) \, d\sigma_t+ \varepsilon^d \int_{\omega}(\nabla u^a)(\varepsilon t) \cdot (\nabla u^b)(\varepsilon t) \, dt \, ,
\end{split}
\]
for all $\varepsilon \in ]-\varepsilon_\#,\varepsilon_\#[\setminus \{0\}$, where $\nu_\omega$  denotes the outward unit normal to $\partial \omega$. 

In order to analyze the  solution to problem \eqref{eq:direpsa} as $\varepsilon \to 0$,  we adapt to the present problem the method developed in Dalla Riva, Musolino, and Rogosin \cite{DaMuRo15} for the solution of the Dirichlet problem in a planar perforated domain, which is based on potential theory. By exploiting some specific integral operators (the single and the double layer potentials) we convert a boundary value problem into a set of boundary integral equations.

To introduce the layer potentials, we denote by  $S_d$  the fundamental solution {of $\Delta\equiv \sum_{j=1}^d\partial_{j}^2$} in $\mathbb{R}^d$: {\it i.e.},
\[
S_d(x)\equiv\frac{1}{(2-d)s_d |x|^{d-2}}\qquad\forall x\in\mathbb{R}^d\setminus\{0\}\,.
\] 
Here $s_d$ denotes the $(d-1)$-dimensional measure of the unit sphere in $\mathbb{R}^d$. Now let $\mathcal{O}$ be an open bounded subset of $\mathbb{R}^d$ of class $C^{1,\alpha}$. If  $\phi\in C^{0,\alpha}(\partial\mathcal{O})$, then we denote by $v[\partial\mathcal{O},\phi]$ the single layer potential with density $\phi$:
\[
v[\partial\mathcal{O},\phi](x)\equiv\int_{\partial\mathcal{O}}\phi(y)S_d(x-y)\,d\sigma_y\qquad\forall x\in\mathbb{R}^d\, .
\] 
It is well known that $v[\partial\mathcal{O},\phi]$ is a continuous function from $\mathbb{R}^d$ to $\mathbb{R}$. The restriction $v^+[\partial\mathcal{O},\phi]\equiv v[\partial\mathcal{O},\phi]_{|\overline{\mathcal{O}}}$ belongs to $C^{1,\alpha}(\overline{\mathcal{O}})$. Moreover, if  we denote by $C^{1,\alpha}_{\mathrm{loc}}(\mathbb{R}^d\setminus\mathcal{O})$  the space of functions on $\mathbb{R}^d\setminus\mathcal{O}$ whose restrictions to $\overline{\mathcal{U}}$ belong to  $C^{1,\alpha}(\overline{\mathcal{U}})$ for all open bounded subsets $\mathcal{U}$ of $\mathbb{R}^d\setminus\mathcal{O}$, then $v^-[\partial\mathcal{O},\phi]\equiv v[\partial\mathcal{O},\phi]_{|\mathbb{R}^d\setminus\mathcal{O}}$  belongs to $C^{1,\alpha}_{\mathrm{loc}}(\mathbb{R}^d\setminus\mathcal{O})$. 

If $\psi\in C^{1,\alpha}(\partial\mathcal{O})$, we introduce  the double layer potential $w[\partial\mathcal{O},\psi]$ by setting
\[
w[\partial\mathcal{O},\psi](x)\equiv-\int_{\partial\mathcal{O}}\psi(y)\;\nu_{\mathcal{O}}(y)\cdot\nabla S_d(x-y)\,d\sigma_y\qquad\forall x\in\mathbb{R}^d\,,
\] 
where $\nu_\mathcal{O}$ denotes the outer unit normal to $\partial\mathcal{O}$. The restriction $w[\partial\mathcal{O},\psi]_{|\mathcal{O}}$ extends to a function $w^+[\partial\mathcal{O},\psi]$ of $C^{1,\alpha}(\overline{\mathcal{O}})$ and  the restriction $w[\partial\mathcal{O},\psi]_{|\mathbb{R}^d\setminus\overline{\mathcal{O}}}$ extends to a function $w^-[\partial\mathcal{O},\psi]$ of $C^{1,\alpha}_{\mathrm{loc}}(\mathbb{R}^d\setminus\mathcal{O})$.

To describe the boundary behavior  of the trace of the double layer potential on $\partial \mathcal{O}$ and of the normal derivative of the single layer potential, we define the boundary integral operators  $W_\mathcal{O}$ and $W^*_\mathcal{O}$:
\[
W_\mathcal{O}[\psi](x)\equiv -\int_{\partial\mathcal{O}}\psi(y)\;\nu_{\mathcal{O}}(y)\cdot\nabla S_d(x-y)\, d\sigma_y\qquad\forall x\in\partial\mathcal{O}\,,
\] for all $\psi\in C^{1,\alpha}(\partial\mathcal{O})$, and
\[
W^*_\mathcal{O}[\phi](x)\equiv \int_{\partial\mathcal{O}}\phi(y)\;\nu_{\mathcal{O}}(x)\cdot\nabla S_d(x-y)\, d\sigma_y\qquad\forall x\in\partial\mathcal{O}\,,
\] for all $\phi\in C^{0,\alpha}(\partial\mathcal{O})$. The operators $W_\mathcal{O}$ and $W^*_\mathcal{O}$  are compact operators  from   $C^{1,\alpha}(\partial\mathcal{O})$  to itself and from 
$C^{0,\alpha}(\partial\mathcal{O})$ to itself, respectively (see Schauder \cite{Sc31, Sc32}). Moreover,  $W_\mathcal{O}$ and $W^*_\mathcal{O}$ are adjoint one to the other with respect to the duality on $C^{1,\alpha}(\partial\mathcal{O})\times C^{0,\alpha}(\partial\mathcal{O})$ induced by the inner product of the Lebesgue space $L^2(\partial\mathcal{O})$ (cf., {\it e.g.}, Kress \cite[Chap.~4]{Kr14}).  We refer the reader to Kress \cite{Kr14} and Wendland \cite{We67,We70}, for the theory of dual systems and the corresponding Fredholm Alternative Principle. Moreover,
\begin{align*}
w^\pm[\partial\mathcal{O},\psi]_{|\partial\mathcal{O}}&=\pm\frac{1}{2}\psi+W_\mathcal{O}[\psi]&\forall\psi\in C^{1,\alpha}(\partial\mathcal{O})\,,\\
\nu_\mathcal{O}\cdot\nabla v^\pm[\partial\mathcal{O},\phi]_{|\partial\mathcal{O}}&=\mp\frac{1}{2}\phi+W^*_\mathcal{O}[\phi]&\forall\phi\in C^{0,\alpha}(\partial\mathcal{O})
\end{align*}
(see, {\it e.g.}, Folland \cite[Chap.~3]{Fo95}).

Finally, we shall need to consider the subspaces of $C^{0,\alpha}(\partial \mathcal{O})$ and of $C^{1,\alpha}(\partial \mathcal{O})$, consisting of functions with zero integral on $\partial \mathcal{O}$:
\begin{equation}\label{eq:_0}
C^{k,\alpha}(\partial \mathcal{O})_0\equiv \Bigg\{f \in C^{k,\alpha}(\partial \mathcal{O})\colon \int_{\partial \mathcal{O}}f\, d\sigma=0\Bigg\} \qquad \text{for $k=0,1$}\, .
\end{equation}

\subsection{An integral formulation of the boundary value problem}

To convert problem \eqref{eq:direpsa} into integral equations, we follow the idea of Lanza de Cristoforis \cite{La08} and of  Dalla Riva, Musolino, and Rogosin \cite{DaMuRo15}. Therefore, we proceed as in  \cite{DaMuRo15} and \cite{AbBoLeMu21, AbLeMu22} and we introduce the map $M\equiv(M^o,M^i,M^c)$  from $]-\varepsilon_\#,\varepsilon_\#[\times C^{0,\alpha}(\partial\Omega)\times C^{0,\alpha}(\partial\omega)$ to $C^{0,\alpha}(\partial\Omega)\times C^{0,\alpha}(\partial\omega)_0\times\mathbb{R}$   by setting
\begin{align*}
&M^o[\varepsilon,\rho^o,\rho^i](x) \equiv \frac{1}{2}\rho^o(x)+W^*_{\Omega}[\rho^o](x)+\int_{\partial\omega}\rho^i(s)\;\nu_{\Omega}(x)\cdot\nabla S_d(x-\varepsilon s)\, d\sigma_s&\forall x\in\partial\Omega\,,\\
&M^i[\varepsilon,\rho^o,\rho^i](t)\equiv\frac{1}{2}\rho^i(t)-W^*_{\omega}[\rho^i](t)-\varepsilon^{d-1} \int_{\partial\Omega}\rho^o(y)\;\nu_{\omega}(t)\cdot\nabla S_d(\varepsilon t- y)\, d\sigma_y&\forall t\in\partial\omega\,,\\
&M^c[\varepsilon,\rho^o,\rho^i]\equiv \int_{\partial\omega}\rho^i\, d\sigma-1\,,
\end{align*} 
for all $(\varepsilon,\rho^o,\rho^i) \in ]-\varepsilon_\#,\varepsilon_\#[\times C^{0,\alpha}(\partial\Omega)\times C^{0,\alpha}(\partial\omega)$. The space $C^{0,\alpha}(\partial\omega)_0$ is defined as in equation \eqref{eq:_0}, {\it i.e.},
\[
C^{0,\alpha}(\partial \omega)_0\equiv \Bigg\{f \in C^{0,\alpha}(\partial \omega)\colon \int_{\partial \omega}f\, d\sigma=0\Bigg\} \, .
\]

Then we can prove the following result of Lanza de Cristoforis \cite[\S3]{La08} (see also  Dalla Riva, Musolino, and Rogosin \cite[Prop.~4.1]{DaMuRo15}).

\begin{prop}\label{rhoeps}
The following statements hold.
\begin{itemize}
\item[(i)] The map $M$ is real analytic. 
\item[(ii)] If $\varepsilon\in[0,\varepsilon_\#[$, then there exists a unique pair $(\rho^o_{\geq}[\varepsilon],\rho^i_{\geq}[\varepsilon])\in C^{0,\alpha}(\partial\Omega)\times C^{0,\alpha}(\partial\omega)$ such that $M[\varepsilon,\rho^o_{\geq}[\varepsilon],\rho^i_{\geq}[\varepsilon]]=0$. \item[(iii)] There exist $\tilde{\varepsilon}_\rho\in ]0,\varepsilon_\#[$ and a real analytic map $(\rho^o[\cdot],\rho^i[\cdot])$ from $]-\tilde{\varepsilon}_\rho,\tilde{\varepsilon}_\rho[$ to $C^{0,\alpha}(\partial\Omega)\times C^{0,\alpha}(\partial\omega)$ such that
\[
M[\varepsilon,\rho^o[\varepsilon],\rho^i[\varepsilon]]=0 \qquad \forall \varepsilon \in ]-\tilde{\varepsilon}_\rho,\tilde{\varepsilon}_\rho[\, .
\]
In particular,
\[
(\rho^o[\varepsilon],\rho^i[\varepsilon])=(\rho^o_{\geq}[\varepsilon],\rho^i_{\geq}[\varepsilon])\qquad \forall \varepsilon \in [0,\tilde{\varepsilon}_\rho[\, .
\]
\end{itemize}
\end{prop}

As explained in \cite{La08}, to represent the solution $u_\varepsilon$ to problem \eqref{eq:direpsa}, we need to consider a further operator and thus we define the map
$\Lambda\equiv(\Lambda^o,\Lambda^i)$  from $]-\varepsilon_\#,\varepsilon_\#[\times  C^{1,\alpha}(\partial\Omega)\times C^{1,\alpha}(\partial\omega)_0$ to $C^{1,\alpha}(\partial\Omega)\times C^{1,\alpha}(\partial\omega)$  by
\begin{align*}
&\Lambda^o[\varepsilon,\theta^o,\theta^i](x)\equiv\frac{1}{2}\theta^o(x)+W_{\Omega}[\theta^o](x) \\
&\qquad\qquad\qquad\qquad +\varepsilon^{d-1}\int_{\partial\omega}\theta^i(s)\;\nu_{\omega}(s)\cdot\nabla S_d(x-\varepsilon s)\, d\sigma_s&\forall x\in\partial\Omega\,,\\
&\Lambda^i[\varepsilon,\theta^o,\theta^i](t)\equiv\frac{1}{2}\theta^i(t)-W_{\omega}[\theta^i](t)+w[\partial\Omega,\theta^o](\varepsilon t)\\
\nonumber
&\qquad\qquad\qquad\qquad -U^a_\#[\varepsilon](t)+\int_{\partial\omega}U^a_\#[\varepsilon]\rho^i[\varepsilon]\,d\sigma&\forall t\in\partial\omega\, ,
\end{align*} 
for all $(\varepsilon,\theta^o,\theta^i) \in ]-\varepsilon_\#,\varepsilon_\#[\times  C^{1,\alpha}(\partial\Omega)\times C^{1,\alpha}(\partial\omega)_0$. Then we have the following result of Lanza de Cristoforis \cite[\S4]{La08} on  $\Lambda$ (cf.~Dalla Riva, Musolino, and Rogosin \cite[Prop.~4.3]{DaMuRo15}).

\begin{prop}\label{thetaeps} 
The following statements hold.
\begin{itemize}
\item[(i)] The map $\Lambda$ is real analytic.
\item[(ii)] If $\varepsilon\in [0,\varepsilon_\#[$, then there exists a unique pair $(\theta^o_\geq[\varepsilon],\theta^i_\geq[\varepsilon])\in C^{1,\alpha}(\partial\Omega)\times C^{1,\alpha}(\partial\omega)_0$ such that $\Lambda[\varepsilon,\theta^o_\geq[\varepsilon],\theta^i_\geq[\varepsilon]]=0$.
\item[(iii)] There exist $\tilde{\varepsilon}_\theta \in ]0,\varepsilon_\#[$ and a real analytic map $(\theta^o[\cdot],\theta^i[\cdot])$ from $]-\tilde{\varepsilon}_\theta,\tilde{\varepsilon}_\theta[$ to $C^{1,\alpha}(\partial\Omega)\times C^{1,\alpha}(\partial\omega)_0$ such that
\[
\Lambda[\varepsilon,\theta^o[\varepsilon],\theta^i[\varepsilon]]=0 \qquad \forall \varepsilon \in ]-\tilde{\varepsilon}_\theta,\tilde{\varepsilon}_\theta[\, .
\]
In particular,
\[
(\theta^o[\varepsilon],\theta^i[\varepsilon])=(\theta^o_{\geq}[\varepsilon],\theta^i_{\geq}[\varepsilon])\qquad \forall \varepsilon \in [0,\tilde{\varepsilon}_\theta[\, .
\]
\end{itemize}
\end{prop}

We now set
\[
\varepsilon_0\equiv \min \{\tilde{\varepsilon}_\rho,\tilde{\varepsilon}_\theta\}\, .
\]

By summing suitable double and single layer potentials, by a modification of the argument of Dalla Riva, Musolino, and Rogosin \cite[Prop.~4.5]{DaMuRo15}, we can represent  the rescaled function $u_\varepsilon(\varepsilon t)$ by means of the functions $\rho^o[\varepsilon]$, $\rho^i[\varepsilon]$, $\theta^o[\varepsilon]$, and $\theta^i[\varepsilon]$ introduced in Propositions \ref{rhoeps} and \ref{thetaeps} (see also Lanza de Cristoforis \cite[\S5]{La08} and  Dalla Riva, Musolino, and Rogosin \cite[\S 2.4]{DaMu13}).

\begin{prop}\label{solution}
Let $\varepsilon\in]0,\varepsilon_0[$. Then
\[
\begin{split}
u^a_\varepsilon(\varepsilon t)& \equiv w^+[\partial\Omega,\theta^o[\varepsilon]](\varepsilon t)-w^-[\partial\omega,\theta^i[\varepsilon]](t)\\
&
+\int_{\partial\omega}U^a_\#[\varepsilon]\rho^i[\varepsilon]\,d\sigma \biggl(v^+[\partial\Omega,\rho^o[\varepsilon]](\varepsilon t) +\varepsilon^{-(d-2)}v^-[\partial\omega,\rho^i[\varepsilon]](t)\biggr)\\
&
\times\biggl(\frac{1}{\int_{\partial\omega}d\sigma}\int_{\partial\omega}v[\partial\Omega,\rho^o[\varepsilon]](\varepsilon s)+\varepsilon^{-(d-2)}v[\partial\omega,\rho^i[\varepsilon]](s)\,d\sigma_s\biggr)^{-1}\\
& = w^+[\partial\Omega,\theta^o[\varepsilon]](\varepsilon t)-w^-[\partial\omega,\theta^i[\varepsilon]](t)\\
&
+\int_{\partial\omega}U^a_\#[\varepsilon]\rho^i[\varepsilon]\,d\sigma \biggl(\varepsilon^{d-2}v^+[\partial\Omega,\rho^o[\varepsilon]](\varepsilon t) +v^-[\partial\omega,\rho^i[\varepsilon]](t)\biggr)\\
&
\times\biggl( \frac{1}{\int_{\partial\omega}d\sigma}\int_{\partial\omega}\varepsilon^{d-2}v[\partial\Omega,\rho^o[\varepsilon]](\varepsilon s)+v[\partial\omega,\rho^i[\varepsilon]](s)\,d\sigma_s\biggr)^{-1}
\end{split}
\] 
for all $t\in\overline{(\varepsilon^{-1}\Omega)}\setminus\omega$.
\end{prop}

\section{Power series expansions of the auxiliary functions  $(\rho^o[\varepsilon],\rho^i[\varepsilon])$ and $(\theta^o[\varepsilon],\theta^i[\varepsilon])$ around $\varepsilon=0$}\label{sec3}

As in \cite{AbLeMu22}, now the plan is to construct an expansion for $\nu_{\omega}(t)\cdot \nabla \Big(u^a_\varepsilon(\varepsilon t)\Big) u^b(\varepsilon t)$ and then to integrate it on $\partial \omega$. To do so, we first obtain a representation of  the auxiliary {density} functions  $(\rho^o[\varepsilon],\rho^i[\varepsilon])$ and $(\theta^o[\varepsilon],\theta^i[\varepsilon])$ and then we use it to represent $\nu_{\omega}(t)\cdot \nabla \Big(u^a_\varepsilon(\varepsilon t)\Big) u^b(\varepsilon t)$.

To compute the coefficients of the series involved in the representation,  we need explicit expressions for the derivatives with respect to $\varepsilon$ of functions of the type $F(\varepsilon x)$. More precisely, we will exploit the equality
\begin{equation}\label{der.eq0}
\partial_\varepsilon^j(F(\varepsilon x))=\sum_{\substack{\beta \in \mathbb{N}^d\\|\beta|=j}}\frac{j!}{\beta!}x^\beta(D^{\beta}F)(\varepsilon x)
\end{equation} which holds for all ${j\in\mathbb{N}}$, $\varepsilon\in\mathbb{R}$, $x\in\mathbb{R}^d$, and for all functions $F$ analytic in a neighbourhood of $\varepsilon x$. Here, if $\alpha\in\mathbb{N}^d$, then $(D^\alpha F)(y)$ denotes the partial derivative of multi-index $\alpha$ with respect to $x$ of the function $F(x)$ evaluated at $y\in\mathbb{R}^d$. {We also exploit formulas for the derivatives of layer potentials in the singularly perturbed set $\Omega \setminus \varepsilon \overline{\omega}$ with respect to the parameter $\varepsilon$ as it is done in Dalla Riva, Luzzini and Musolino \cite{DaLuMu22}.}

Then we have the following  variant of  Dalla Riva, Musolino, and Rogosin \cite[Prop.~5.1]{DaMuRo15}, where we represent $(\rho^o[\varepsilon],\rho^i[\varepsilon])$ as a power series for $\varepsilon$ close to $0$.

\begin{prop}\label{rhok} Let $(\rho^o[\varepsilon],\rho^i[\varepsilon])$ be as in Proposition \ref{rhoeps} for all $\varepsilon\in]-\varepsilon_0,\varepsilon_0[$. Then there exist $\varepsilon_\rho\in]0,\varepsilon_0[$ and a sequence $\{(\rho^o_k,\rho^i_k)\}_{k\in\mathbb{N}}$ in $C^{0,\alpha}(\partial\Omega)\times C^{0,\alpha}(\partial\omega)$ such that 
\begin{equation}\label{rhok0}
\rho^o[\varepsilon]=\sum_{k=0}^{+\infty}\frac{\rho^o_k}{k!}\varepsilon^k\quad\text{ and }\quad\rho^i[\varepsilon]=\sum_{k=0}^{+\infty}\frac{\rho^i_k}{k!}\varepsilon^k\qquad \forall \varepsilon\in]-\varepsilon_\rho,\varepsilon_\rho[\,,
\end{equation}
where the two series converge normally  in $C^{0,\alpha}(\partial\Omega)$ and in $C^{0,\alpha}(\partial\omega)$, respectively, for $\varepsilon\in]-\varepsilon_\rho,\varepsilon_\rho[$. Moreover,  the pair of functions $(\rho^o_0,\rho^i_0)$ is the unique solution in $C^{0,\alpha}(\partial\Omega)
\times C^{0,\alpha}(\partial\omega)$ of the following system of integral equations
\begin{align}\nonumber
&\frac{1}{2}\rho^o_0(x)+W^*_{\Omega}[\rho^o_0](x)=-\nu_{\Omega}(x)\cdot\nabla S_d(x) & \forall x\in\partial\Omega\,,\\
\nonumber
&\frac{1}{2} \rho^i_0(t)-W^*_{\omega}[ \rho^i_0](t)=0 & \forall t\in\partial\omega\,,\\
\nonumber
&\int_{\partial\omega}\rho^i_0\, d\sigma=1\,, 
\end{align}
for each $k\in\{1,\dots,d-2\}$ the pair of functions $(\rho^o_k,\rho^i_k)$ is the unique solution in $C^{0,\alpha}(\partial\Omega)
\times C^{0,\alpha}(\partial\omega)$ of the following system of integral equations
\begin{align}\nonumber
&\frac{1}{2}\rho^o_k(x)+W^*_{\Omega}[\rho^o_k](x)\\
\nonumber
&\quad =(-1)^{k+1}\sum_{\substack{\beta \in \mathbb{N}^d\\|\beta|=k}}\frac{k!}{\beta!}\nu_{\Omega}(x)\cdot(\nabla D^\beta S)(x)\int_{\partial\omega}\rho^i_{0}(s)s^\beta\, d\sigma_s& \forall x\in\partial\Omega\,,\\
\nonumber
&\rho^i_k(t)=0 \qquad \forall t \in \partial \omega\, ,
\end{align}  
and for each $k\in\mathbb{N}\setminus\{0,\dots,d-2\}$ the pair $(\rho^o_k,\rho^i_k)$ is the unique solution in $C^{0,\alpha}(\partial\Omega)
\times C^{0,\alpha}(\partial\omega)$ of the following system of integral equations which involves $\{(\rho^o_j,\rho^i_j)\}_{j=0}^{k-1}$,
\begin{align}\nonumber
&\frac{1}{2}\rho^o_k(x)+W^*_{\Omega}[\rho^o_k](x)\\
\nonumber
&\quad =\sum_{j=0}^{k}\binom{k}{j}(-1)^{j+1}\sum_{\substack{\beta \in \mathbb{N}^d\\|\beta|=j}}\frac{j!}{\beta!}\nu_{\Omega}(x)\cdot(\nabla D^\beta S)(x)\int_{\partial\omega}\rho^i_{k-j}(s)s^\beta\, d\sigma_s \qquad \qquad \forall x\in\partial\Omega\,,\\
\nonumber
&\frac{1}{2} \rho^i_k(t)-W^*_{\omega}[ \rho^i_k](t)\\
\nonumber
&\quad =\frac{k!}{(k-(d-1))!}\sum_{j=0}^{k-(d-1)}\binom{k-(d-1)}{j}  \sum_{\substack{\beta \in \mathbb{N}^d\\|\beta|=j}}\frac{j!}{\beta!}(-1)^{j+1}t^\beta \nu_{\omega}(t)\cdot \int_{\partial\omega}\rho^o_{k-(d-1)-j} \;  \nabla D^\beta S_d\, d\sigma \\
\nonumber &\qquad \qquad  \qquad \qquad\forall t\in\partial\omega\,,\\
\nonumber
&\int_{\partial\omega}\rho^i_k\, d\sigma=0\,.
\end{align}  
\end{prop}
\begin{proof}

By the real analyticity result of Proposition \ref{rhoeps} (iii) for the map
\[
\varepsilon \mapsto (\rho^o[\varepsilon],\rho^i[\varepsilon])\, ,
\] 
we deduce that there exist $\varepsilon_\rho$ and  $\{(\rho^o_k,\rho^i_k)\}_{k\in\mathbb{N}}$ such that \eqref{rhok0} holds. By the real analyticity of $\varepsilon \mapsto (\rho^o[\varepsilon],\rho^i[\varepsilon])$, we have $(\rho^o_k,\rho^i_k)=(\partial_\varepsilon^k\rho^o[0],\partial_\varepsilon^k\rho^i[0])$ for all   $k\in\mathbb{N}$.  Therefore our goal is to identify the derivatives $(\partial_\varepsilon^k\rho^o[0],\partial_\varepsilon^k\rho^i[0])$ for all   $k\in\mathbb{N}$. Equality $M[\varepsilon,\rho^o[\varepsilon],\rho^i[\varepsilon]]=0$ for all $\varepsilon\in]-\varepsilon_0,\varepsilon_0[$ (cf.~Proposition \ref{rhoeps} (ii)) implies that the map
\[
\varepsilon \mapsto M[\varepsilon,\rho^o[\varepsilon],\rho^i[\varepsilon]]
\] 
has zero derivatives, {\it i.e.},   
\begin{equation}\label{eq:Mder0}
\partial_\varepsilon^k(M[\varepsilon,\rho^o[\varepsilon],\rho^i[\varepsilon]])=0 \qquad \forall \varepsilon\in]-\varepsilon_0,\varepsilon_0[\, ,k\in\mathbb{N}\, .
\end{equation} 
Therefore we compute $ \partial_\varepsilon^k(M[\varepsilon,\rho^o[\varepsilon],\rho^i[\varepsilon]])$ and we have
\begin{align}\label{rhook'}
&\partial_\varepsilon^k(M^o[\varepsilon,\rho^o[\varepsilon],\rho^i[\varepsilon]])(x)=\frac{1}{2}\partial_\varepsilon^k \rho^o[\varepsilon](x)+W^*_{\Omega}[\partial_\varepsilon^k \rho^o[\varepsilon]](x)\\
\nonumber
&+\sum_{j=0}^{k}\binom{k}{j}(-1)^j\sum_{\substack{\beta \in \mathbb{N}^d\\|\beta|=j}}\frac{j!}{\beta!}\nu_{\Omega}(x)\cdot\int_{\partial\omega}\partial_\varepsilon^{k-j} \rho^i[\varepsilon](s)\;s^{\beta}(\nabla D^\beta S_d)(x-\varepsilon s)\, d\sigma_s\\ 
\nonumber
&\qquad\qquad\qquad\qquad\qquad\qquad\qquad\qquad\qquad\qquad\qquad\qquad\qquad\qquad\qquad \forall x\in\partial\Omega\,,\\
\label{rhoik1'}
&\partial_\varepsilon^k(M^i[\varepsilon,\rho^o[\varepsilon],\rho^i[\varepsilon]])(t)=\frac{1}{2}\partial_\varepsilon^k \rho^i[\varepsilon](t)-W^*_{\omega}[\partial_\varepsilon^k \rho^i[\varepsilon]](t)\\
\nonumber
\nonumber
&- \partial_\varepsilon^k\bigg(\varepsilon^{d-1} \int_{\partial\Omega}{\rho^o[\varepsilon](y)}\;\nu_{\omega}(t)\cdot\nabla S_d(\varepsilon t- y)\, d\sigma_y\bigg)\\
\nonumber
&\qquad\qquad\qquad\qquad\qquad\qquad\qquad\qquad\qquad\qquad\qquad\qquad\qquad\qquad\qquad \forall t\in\partial\omega\,,\\
\label{rhoik2'}
&\partial_\varepsilon^k (M^c[\varepsilon,\rho^o[\varepsilon],\rho^i[\varepsilon]])=\int_{\partial\omega}\partial_\varepsilon^k\rho^i[\varepsilon]\, d\sigma-\delta_{{0k}}\,,
\end{align} for all $\varepsilon\in]-\varepsilon_0,\varepsilon_0[$ and all $k\in\mathbb{N}$ (see also \eqref{der.eq0}). {Here above $\delta_{ij}$ denotes the Kronecker delta function.} 
Next we note that if $k<d-1$ then
\[
\partial_\varepsilon^k\bigg(\varepsilon^{d-1}\int_{\partial\Omega}\rho^o{[\varepsilon]}(y)\;\nu_{\omega}(t)\cdot\nabla S_d(\varepsilon t- y)\, d\sigma_y\bigg)_{|\varepsilon=0}=0 \qquad \forall t \in \partial \omega\, . 
\]
Instead,  if $k\geq d-1$ then
\[
\begin{split}
&\partial_\varepsilon^k\bigg(\varepsilon^{d-1}\int_{\partial\Omega}\rho^o[\varepsilon](y)\;\nu_{\omega}(t)\cdot\nabla S_d(\varepsilon t- y)\, d\sigma_y\bigg)_{|\varepsilon=0}\\
&=\binom{k}{(d-1)}(d-1)! \partial_\varepsilon^{k-(d-1)}\bigg(\int_{\partial\Omega}\rho^o[\varepsilon](y)\;\nu_{\omega}(t)\cdot\nabla S_d(\varepsilon t- y)\, d\sigma_y\bigg)_{|\varepsilon=0}\\
&=\frac{k!}{(k-(d-1))!}\sum_{j=0}^{k-(d-1)}\binom{k-(d-1)}{j}\\
&\qquad \qquad \times \sum_{\substack{\beta \in \mathbb{N}^d\\|\beta|=j}}\frac{j!}{\beta!} \int_{\partial\omega}\partial_\varepsilon^{k-(d-1)-j}\big(\rho^o[\varepsilon](y)\big)_{|\varepsilon=0}\; t^\beta \nu_{\omega}(t)\cdot\nabla D^\beta S_d(-y)\, d\sigma_y\\
&=\frac{k!}{(k-(d-1))!}\sum_{j=0}^{k-(d-1)}\binom{k-(d-1)}{j}\\
&\qquad \qquad \times \sum_{\substack{\beta \in \mathbb{N}^d\\|\beta|=j}}\frac{j!}{\beta!}(-1)^{|\beta|+1} \int_{\partial\omega}\partial_\varepsilon^{k-(d-1)-j}\big(\rho^o[\varepsilon](y)\big)_{|\varepsilon=0}\; t^\beta \nu_{\omega}(t)\cdot\nabla D^\beta S_d(y)\, d\sigma_y \qquad \forall t \in \partial \omega\, . 
\end{split}
\]
Then, by taking $\varepsilon=0$  in \eqref{rhook'}--\eqref{rhoik2'} and by equality \eqref{eq:Mder0}, we deduce that for each $k \in \mathbb{N}$ the pair of functions $(\rho^o_k,\rho^i_k)$ is a solution of the corresponding integral equations of the statement. By the characterization of the kernels of the operators of Dalla Riva, Lanza de Cristoforis, and Musolino \cite[\S 6.5 and \S 6.6]{DaLaMu21}, we deduce the uniqueness of the solutions of the integral equations of the statement. 
\end{proof}

By a similar computation, in the following proposition we construct the power series expansion of $(\theta^o[\varepsilon],\theta^i[\varepsilon])$.

\begin{prop}\label{thetak} Let $(\theta^o[\varepsilon],\theta^i[\varepsilon])$ be as in Proposition \ref{thetaeps} for all $\varepsilon\in]-\varepsilon_0,\varepsilon_0[$. Then there exist $\varepsilon_\theta\in]0,\varepsilon_0[$ and a sequence $\{(\theta^o_k,\theta^i_k)\}_{k\in\mathbb{N}}$ in $C^{1,\alpha}(\partial\Omega)\times C^{1,\alpha}(\partial\omega)_0$ such that 
\begin{equation}\label{thetak0}
\theta^o[\varepsilon]=\sum_{k=0}^\infty\frac{\theta^o_k}{k!}\varepsilon^k\quad\text{and }\quad\theta^i[\varepsilon]=\sum_{k=0}^\infty\frac{\theta^i_k}{k!}\varepsilon^k\qquad\forall \varepsilon\in]-\varepsilon_\theta,\varepsilon_\theta[\,,
\end{equation} where the two series converge normally in $C^{1,\alpha}(\partial\Omega)$ and in $C^{1,\alpha}(\partial\omega)_0$, respectively, for $\varepsilon\in]-\varepsilon_\theta,\varepsilon_\theta[$. Moreover,  
\[
(\theta^o_0,\theta^i_0)=(0,0)\, , \qquad \theta^o_k=0\qquad \forall k \in \{0,\dots, d-1\}\, ,
\]
for each  $k \in \{0,\dots, d-1\}$ the function $\theta^i_{k}$ is the unique solution in $C^{1,\alpha}(\partial\omega)_0$ of
\begin{equation}\label{thetai1}
\frac{1}{2}\theta^i_k(t)-W_{\omega}[\theta^i_k](t)=\sum_{|\beta|=k}\frac{k!}{\beta!}t^\beta (D^{\beta}u^a)(0)-\sum_{l=0}^k \sum_{\substack{\beta \in \mathbb{N}^d\\ |\beta|=l}} \binom{k}{l}\frac{l!}{\beta!}{\int_{\partial\omega}s^{\beta}(D^\beta u^a)(0)\rho^i_{k-l}(s) \, d\sigma_s}\quad \forall t\in\partial\omega\,{,}
\end{equation}
and for each $k\in\mathbb{N}\setminus\{0,\dots, d-1\}$ the pair $(\theta^o_k,\theta^i_k)$ is the unique solution in $C^{1,\alpha}(\partial\Omega)\times C^{1,\alpha}(\partial\omega)_0$ of the following  system of integral equations which involves $\{(\theta^o_j,\theta^i_j)\}_{j=0}^{k-1}$,
\begin{align}\label{thetaok}
&\frac{1}{2}\theta^o_k(x)+W_{\Omega}[\theta^o_k](x)\\
\nonumber
&=\frac{k!}{(k-(d-1))!}\sum_{j=0}^{k-(d-1)-1}\binom{k-(d-1)}{j}(-1)^{j+1}\sum_{|\beta|=j}\frac{j!}{\beta!}(\nabla D^{\beta} S_d)(x)\cdot \int_{\partial\omega}\theta^i_{k-(d-1)-j}(s)\;\nu_{\omega}(s)s^{\beta} d\sigma_s\\ 
\nonumber
&\qquad\qquad\qquad\qquad\qquad\qquad\qquad\qquad\qquad\qquad\qquad\qquad\qquad\qquad\qquad \forall x\in\partial\Omega\,,\\
\label{thetaik}
&\frac{1}{2}\theta^i_k(t)-W_{\omega}[\theta^i_k](t)=\sum_{j=0}^{k-(d-1)}\binom{k}{j}(-1)^{j+1}\sum_{|\beta|=j}\frac{j!}{\beta!}t^\beta\int_{\partial\Omega}\theta^o_{k-j}\nu_{\Omega}\cdot \nabla D^{\beta} S_d \,d\sigma\\ \nonumber
&+\sum_{|\beta|=k}\frac{k!}{\beta!}t^\beta (D^{\beta}u^a)(0)-\sum_{l=0}^k \sum_{\substack{\beta \in \mathbb{N}^d\\ |\beta|=l}} \binom{k}{l}\frac{l!}{\beta!}{\int_{\partial\omega}s^{\beta}(D^\beta u^a)(0)\rho^i_{k-l}(s) \, d\sigma_s}\qquad  \forall t\in\partial\omega\,.\nonumber
\end{align}
\end{prop}
\proof  We proceed as in  Dalla Riva, Musolino, and Rogosin \cite[Prop.~5.2]{DaMuRo15} and in \cite[Prop~3.2]{AbBoLeMu21}. Proposition \ref{thetaeps} (iii) implies that the map  
\[
\varepsilon \mapsto (\theta^o[\varepsilon],\theta^i[\varepsilon])
\]  
is real analytic. As a consequence, there exist $\varepsilon_\theta$ and $\{(\theta^o_k,\theta^i_k)\}_{k\in\mathbb{N}}$ such that \eqref{thetak0} holds.  Our goal is to identify the terms $(\theta^o_k,\theta^i_k)$ for each ${k\in\mathbb{N}}$ and we do that by computing the derivatives of $(\theta^o[\varepsilon],\theta^i[\varepsilon])$ with respect to $\varepsilon$. By Proposition \ref{thetaeps} (ii), we have 
\[
\Lambda[\varepsilon,\theta^o[\varepsilon],\theta^i[\varepsilon]]=0 \qquad \forall \varepsilon\in]-\varepsilon_0,\varepsilon_0[\, .
\] 
We can compute the derivative with respect to $\varepsilon$ in the equality above and we deduce that 
\begin{equation}\label{eq:Lmbdder0}
\partial_\varepsilon^k(\Lambda[\varepsilon,\theta^o[\varepsilon],\theta^i[\varepsilon]])=0 \qquad \forall \varepsilon\in]-\varepsilon_0,\varepsilon_0[\, , \forall k\in\mathbb{N}\, .
\end{equation}
Hence we compute $\partial_\varepsilon^k(\Lambda[\varepsilon,\theta^o[\varepsilon],\theta^i[\varepsilon]])$ and we have
\begin{align}\label{thetaok'}
&\partial_\varepsilon^k(\Lambda^o[\varepsilon,\theta^o[\varepsilon],\theta^i[\varepsilon]])(x)=\frac{1}{2}\partial^k_\varepsilon\theta^o[\varepsilon](x)+W_{\Omega}[\partial^k_\varepsilon\theta^o[\varepsilon]](x)\\
\nonumber
&+\partial_\varepsilon^k(\varepsilon^{d-1}\int_{\partial\omega}\theta^i[\varepsilon](s)\;\nu_{\omega}(s)\cdot\nabla S_d(x-\varepsilon s)\, d\sigma_s) \qquad\qquad\qquad \forall x\in\partial\Omega\,,
\end{align}
\begin{align}
\label{thetaik'}
&\partial_\varepsilon^k(\Lambda^i[\varepsilon,\theta^o[\varepsilon],\theta^i[\varepsilon]])(t)=\frac{1}{2}\partial^k_\varepsilon\theta^i[\varepsilon](t)-W_{\omega}[\partial^k_\varepsilon\theta^i[\varepsilon]](t)\\
\nonumber
&-\sum_{j=0}^k\binom{k}{j}\sum_{|\beta|=j}\frac{j!}{\beta!}t^{\beta}\int_{\partial\Omega}\partial_\varepsilon^{k-j}\theta^o[\varepsilon](y)\,\nu_{\Omega}(y)\cdot(\nabla D^{\beta} S_d)(\varepsilon t-y)\,d\sigma_y\\
\nonumber
&-\sum_{|\beta|=k}\frac{k!}{\beta!}t^{\beta} (D^{\beta}u^a)(\varepsilon t)\\
&+\sum_{l=0}^k \sum_{\substack{\beta \in \mathbb{N}^d\\ |\beta|=l}} \binom{k}{l}\frac{l!}{\beta!}{ \int_{\partial\omega}s^{\beta}(D^\beta u^a)(\varepsilon s)\partial_\varepsilon^{k-l}\rho^i[\varepsilon](s)\,d\sigma_s}\qquad\qquad\qquad\qquad\forall t\in\partial\omega\,, \nonumber
\end{align}
for all $\varepsilon\in]-\varepsilon_0,\varepsilon_0[$ and all  $k\in\mathbb{N}$.  
Then we note that if $k<d-1$ then
\[
\partial_\varepsilon^k\bigg(\varepsilon^{d-1}\int_{\partial\omega}\theta^i[\varepsilon](s)\;\nu_{\omega}(s)\cdot\nabla S_d(x-\varepsilon s)\, d\sigma_s\bigg)_{|\varepsilon=0}=0 \qquad \forall x \in \partial \Omega\, . 
\]
Instead,  if $k\geq d-1$ then
\[
\begin{split}
&\partial_\varepsilon^k\bigg(\varepsilon^{d-1}\int_{\partial\omega}\theta^i[\varepsilon](s)\;\nu_{\omega}(s)\cdot\nabla S_d(x-\varepsilon s)\, d\sigma_s\bigg)_{|\varepsilon=0}\\
&=\binom{k}{(d-1)}(d-1)! \partial_\varepsilon^{k-(d-1)}\bigg(\int_{\partial\omega}\theta^i[\varepsilon](s)\;\nu_{\omega}(s)\cdot\nabla S_d(x-\varepsilon s)\, d\sigma_s\bigg)_{|\varepsilon=0}\\
&=\frac{k!}{(k-(d-1))!}\sum_{j=0}^{k-(d-1)}\binom{k-(d-1)}{j}(-1)^j\\
&\qquad \qquad \times \sum_{|\beta|=j}\frac{j!}{\beta!} \int_{\partial\omega}\partial_\varepsilon^{k-(d-1)-j}\big(\theta^i[\varepsilon](s)\big)_{|\varepsilon=0}\; s^\beta \nu_{\omega}(s)\cdot\nabla D^\beta S_d(x)\, d\sigma_s \qquad \forall x \in {\partial}\Omega\, . 
\end{split}
\]
The real analyticity of $\varepsilon \mapsto (\theta^o[\varepsilon],\theta^i[\varepsilon])$ implies that $(\theta^o_k,\theta^i_k)=(\partial_\varepsilon^k\theta^o[0],\partial_\varepsilon^k\theta^i[0])$ for all  $k\in\mathbb{N}$. Therefore, by taking $\varepsilon=0$  in \eqref{thetaok'} and \eqref{thetaik'} and by equality \eqref{eq:Lmbdder0}, we deduce that  $(\theta^o_0,\theta^i_0)=(0,0)$, that $\theta^o_k=0$ for all $k \in \{0,\dots,d-1\}$, that  for each $k \in \{0,\dots,d-1\}$ the function $\theta^i_k$ solves equation \eqref{thetai1},  and that $(\theta^o_k,\theta^i_k)$  is a solution of \eqref{thetaok} and \eqref{thetaik} for all $k\in\mathbb{N}\setminus \{0,\dots,d-1\}$. To conclude,  we note  that, by classical potential theory (cf., {\it e.g.}, Dalla Riva, Lanza de Cristoforis, and Musolino \cite[\S 6.5 and \S 6.6]{DaLaMu21}), equation \eqref{thetai1} has a unique solution in $C^{1,\alpha}(\partial\omega)_0$ and equations \eqref{thetaok}-\eqref{thetaik} have a unique solution in $C^{1,\alpha}(\partial\Omega)\times C^{1,\alpha}(\partial\omega)_0$. 
 \qed

\section{Series expansion of $\nu_{\omega}(\cdot)\cdot \nabla \big(u^a_\varepsilon(\varepsilon \cdot)\big) u^b(\varepsilon \cdot)$ around $\varepsilon=0$}\label{sec4}

The next step is to obtain a series expansion of $\nu_{\omega}(\cdot)\cdot \nabla \big(u^a_\varepsilon(\varepsilon \cdot)\big) u^b(\varepsilon \cdot)$ for $\varepsilon$ close to $0$. The coefficients of such series will be defined by means of the sequences $\{(\rho^o_k,\rho^i_k)\}_{k\in\mathbb{N}}$ and $\{(\theta^o_k,\theta^i_k)\}_{k\in\mathbb{N}}$  introduced in Section \ref{sec3}. We begin with the intermediate result of Proposition \ref{uk} below, whose proof can be obtained by arguing as in the proof of \cite[Prop.~2.9]{AbLeMu22}.

\begin{prop}\label{uk} Let $\{(\rho^o_k,\rho^i_k)\}_{k\in\mathbb{N}}$ and $\{(\theta^o_k,\theta^i_k)\}_{k\in\mathbb{N}}$ be as in Propositions \ref{rhok} and \ref{thetak}, respectively.  Let 
\[
\begin{split}
&u^a_{\mathrm{m},0}(t)\equiv0\qquad\qquad\qquad\qquad\qquad\qquad\qquad\qquad  \forall t\in\mathbb{R}^d\setminus\omega\,, \\
& u^a_{\mathrm{m},k}(t)\equiv- w^-[\partial\omega,\theta^i_k](t) \qquad\qquad\qquad\qquad\qquad  \forall t\in\mathbb{R}^d\setminus\omega\,, k \in \{0,\dots,d-1\}\, , \\
&u^a_{\mathrm{m},k}(t)\equiv\frac{1}{k!}\sum_{j=0}^{k-d}\binom{k}{j}(-1)^{j}\sum_{\substack{\beta \in \mathbb{N}^d\\|\beta|=j}}\frac{j!}{\beta!} t^\beta \int_{\partial\Omega}\theta^o_{k-j}\,\nu_{\Omega}\cdot(\nabla D^\beta S_d)\,d\sigma\\
&\qquad\qquad - \frac{1}{k!}w^-[\partial\omega,\theta^i_k](t) \quad\qquad\qquad\qquad\qquad  \forall t\in\mathbb{R}^d\setminus\omega\,, \quad \forall k \geq d \\
\end{split}
\]
and
 \[
\begin{split}
&v_{\mathrm{m},k}(t)\equiv\frac{1}{(k-(d-2))!}\sum_{j=0}^{k-(d-2)} \binom{k-(d-2)}{j}(-1)^j\sum_{\substack{\beta \in \mathbb{N}^d\\ {|\beta|=j}}}\frac{j!}{\beta!} t^\beta \int_{\partial\Omega}\rho^o_{k-(d-2)-j}D^\beta S_d\,d\sigma
 +\frac{1}{k!}v^-[\partial\omega,\rho^i_k](t)\\
&\qquad\qquad\qquad\qquad\qquad\qquad\qquad\qquad\qquad\qquad\qquad\qquad\qquad\qquad \forall t\in\mathbb{R}^d\setminus\omega\,, \\
&g^a_k\equiv\frac{1}{k!}\sum_{l=0}^k \sum_{\substack{\beta \in \mathbb{N}^d\\|\beta|=l}} \binom{k}{l}\frac{l!}{\beta!}\int_{\partial\omega}s^\beta(D^\beta u^a)(0)\rho^i_{k-l}(s)\,d\sigma_s\,,\\
&r_k\equiv\frac{1}{(k-(d-2))!\int_{\partial\omega}d\sigma}\sum_{j=0}^{k-(d-2)}\binom{k-(d-2)}{j}(-1)^j\sum_{\substack{\beta \in \mathbb{N}^d\\ |\beta|=j}}\frac{j!}{\beta!} \int_{\partial\omega}s^\beta \,d\sigma_s\int_{\partial\Omega}\rho^o_{k-(d-2)-j}D^\beta S_d\,d\sigma\\
&\qquad+\frac{1}{k!\int_{\partial\omega}d\sigma}\int_{\partial\omega}v[\partial\omega,\rho^i_k]\,d\sigma\,,
\end{split}
\] for all $k\in\mathbb{N}$. Here above the sum $\sum_{j=0}^{k-(d-2)}$ is omitted if $k-(d-2)<0$. {Let} 
\[
\begin{split}
&u^l_{\#,k}(t)\equiv\sum_{\substack{\beta\in \mathbb{N}^d\\ |\beta|=k}}\frac{D^\beta u^l(0)}{\beta !}t^\beta \qquad\qquad\quad  \forall t\in  \mathbb{R}^d\,, \qquad l=a,b\, ,\\
& \tilde{u}_{k}(t)\equiv\sum_{l=0}^k \nu_{\omega}(t)\cdot\nabla u^a_{\mathrm{m},l |\partial \omega}(t) u^b_{\#,k-l}(t)
\qquad  \forall t\in\partial\omega\,, \\
& \tilde{v}_k(t)\equiv \nu_{\omega}(t)\cdot\nabla v_{\mathrm{m},k|\partial \omega}(t) \qquad  \forall t\in\partial\omega\,, \\
&\tilde{g}_{k}(t)\equiv\sum_{l=0}^k g^a_l u^b_{\#,k-l}(t) \quad\qquad\qquad\qquad\qquad  \forall t\in\partial \omega\,, \end{split}
\]
for all $k\in\mathbb{N}$. Then  there exists $\tilde{\varepsilon}\in]0,\varepsilon_0]$  such that
\begin{equation}\label{funepsm}
\nu_{\omega}(\cdot) \cdot \nabla \big(u^a_{\varepsilon}(\varepsilon\cdot)\big)_{|\partial \omega}u^b(\varepsilon \cdot)_{|\partial \omega}=\sum_{k=1}^\infty \tilde{u}_{k}(\cdot)\varepsilon^k+\Bigg(\sum_{k=0}^\infty \tilde{g}_{k}(\cdot)\varepsilon^k\Bigg)\frac{\sum_{k=0}^\infty \tilde{v}_{k}(\cdot)\varepsilon^k}{\sum_{k=0}^\infty r_{k}\varepsilon^k}
\end{equation} for all $\varepsilon\in]-\tilde{\varepsilon},\tilde{\varepsilon}[\setminus\{0\}$. Moreover, the series $\sum_{k=0}^\infty \tilde{g}_{k}\varepsilon^k$, $\sum_{k={1}}^\infty \tilde{u}_{k}\varepsilon^k$, and $\sum_{k=0}^\infty \tilde{v}_{k}\varepsilon^k$ converge normally  in $C^{0,\alpha}(\partial \omega)$  for $\varepsilon\in]-\tilde{\varepsilon},\tilde{\varepsilon}[$ and  $\sum_{k=0}^\infty r_{k}\varepsilon^k$ converge absolutely in $]-\tilde{\varepsilon},\tilde{\varepsilon}[$. 
\end{prop}

We would like to have a representation formula for $\nu_{\omega}(\cdot) \cdot \nabla \big(u^a_{\varepsilon}(\varepsilon\cdot)\big)_{|\partial \omega}u^b(\varepsilon \cdot)_{|\partial \omega}$ in the form of a convergent power series of the type $\sum_{n=0}^\infty  \varphi_\varepsilon(\cdot) \varepsilon^n$. By exploiting an  argument similar to that of  Dalla Riva, Musolino, and Rogosin \cite[Thm.~6.3]{DaMuRo15}, we can prove  Theorem  \ref{umk} below where we obtain from formula \eqref{funepsm}  a series expansion for the map which takes $\varepsilon$ to  $\nu_{\omega}(\cdot) \cdot \nabla \big(u^a_{\varepsilon}(\varepsilon\cdot)\big)_{|\partial \omega}u^b(\varepsilon \cdot)_{|\partial \omega}$ (see also \cite[Thm.~2.10]{AbBoLeMu21} and \cite[Thm.~4.3]{AbLeMu22}).

\begin{theorem}\label{umk}
With the notation introduced in Proposition \ref{uk}, let $\{\tilde{a}_{n}\}_{n\in\mathbb{N}}$ be the sequence of functions from $\partial \omega$ to $\mathbb{R}$ defined by 
\[
\tilde{a}_{n}\equiv\sum_{k=0}^n \tilde{g}_{n-k}\tilde{v}_{k}\qquad\forall n\in\mathbb{N}\,.
\] Let $\{\tilde{\lambda}_{n}\}_{n\in\mathbb{N}}$ be the sequence of functions from $\partial \omega$ to $\mathbb{R}$ defined by 
\[
\tilde{\lambda}_{0}\equiv {\tilde{a}_{0}/{r_0}}\, ,\qquad \tilde{\lambda}_{n}\equiv \tilde{u}_{n} +\tilde{a}_{n}/{r_0}+\sum_{k=1}^n \tilde{a}_{n-k} \sum_{j=1}^k\frac{(-1)^j}{r_0^{j+1}}\sum_{\substack{\beta\in(\mathbb{N}\setminus\{0\})^j\\\ |\beta|=k}}\prod_{h=1}^jr_{\beta_h} \qquad \forall n \geq 1\, .
\] 
Then there exists $\tilde{\varepsilon}'\in]0,\varepsilon_0]$ such that 
  \begin{equation}\label{fuepsmseries}
\nu_{\omega}(\cdot) \cdot \nabla \big(u^a_{\varepsilon}(\varepsilon\cdot)\big)_{|\partial \omega}u^b(\varepsilon \cdot)_{|\partial \omega}=\sum_{n=0}^\infty\tilde{\lambda}_{n}(\cdot) \varepsilon^n
\end{equation} for all $\varepsilon\in]-\tilde{\varepsilon}',\tilde{\varepsilon}'[\setminus\{0\}$. Moreover, the series 
\[
\sum_{n=0}^\infty\tilde{\lambda}_{n}(\cdot) \varepsilon^n
\] 
converges normally in $C^{0,\alpha}(\partial \omega)$  for $\varepsilon \in]-\tilde{\varepsilon}',\tilde{\varepsilon}'[$ {and}
\[
\begin{split}
 {\tilde{\lambda}_{0}=}&\frac{u^a(0)u^b(0)}{r_0} \frac{\partial}{\partial \nu_{\omega}}v^-[\partial \omega, \rho^i_0]\, .
\end{split}
\]
\end{theorem}

\section{Series expansion of $\mathrm{Cap}_\Omega(\varepsilon \overline{\omega},u^a,u^b)$}\label{sec5}

We recall that the $(u^a,u^b)$-capacity $\mathrm{Cap}_\Omega(\varepsilon \overline{\omega},u^a,u^b)$ can be represented  as the sum of $\int_{\Omega_\varepsilon}\nabla u^a_\varepsilon \cdot \nabla u^b_\varepsilon \, dx$ and of $\int_{\varepsilon \omega}\nabla u^a \cdot \nabla u^b \, dx$. Therefore, in order to compute a series expansion of $\mathrm{Cap}_\Omega(\varepsilon \overline{\omega},u^a,u^b)$, we begin by  providing an expansion for $\int_{\varepsilon \omega}\nabla u^a \cdot \nabla u^b \, dx$ around $\varepsilon=0$. 

\begin{lem}\label{lem:nrguom}
Let $\{\xi_{n}\}_{n\in\mathbb{N}}$ be the sequence of real numbers defined by 
\[
\begin{split}
& \xi_0\equiv\xi_1\equiv \dots\equiv \xi_{d-1}\equiv0\, , \qquad\xi_{n}\equiv\sum_{j=1}^d\sum_{l=0}^{n-d} \int_{\omega} \partial_j u^a_{\#,l+1}(t)\partial_j u^b_{\#,n-l-(d-1)}(t)\, dt\qquad \forall n \geq d\ .
\end{split}
\] 
Then there exists $\varepsilon_\xi\in]0,\varepsilon_0]$ such that 
\[
\int_{\varepsilon \omega}\nabla u^a \cdot \nabla u^b \, dx=\sum_{n=d}^\infty\xi_n \varepsilon^n
\]
for all $\varepsilon\in]0,\varepsilon_\xi[$. Moreover, 
\[
\xi_d = \nabla u^a(0) \cdot \nabla u^b(0) m_d(\omega) \, ,
\]
and the series 
\[
\sum_{n=d}^\infty\xi_n \varepsilon^n
\] 
converges absolutely for $\varepsilon \in]-\varepsilon_\xi,\varepsilon_\xi[$. (The symbol $m_d(\dots)$ denotes the d-dimensional Lebesgue measure of a set).
\end{lem}
\proof 
We argue as in the proof of \cite[Lem.~2.12]{AbLeMu22}. We first note that the Theorem of change of variable in integrals implies that
\[
\int_{\varepsilon \omega}\nabla u^a \cdot \nabla u^b \, dx=\varepsilon^d\int_{ \omega}\nabla u^a(\varepsilon t)\cdot \nabla u^b(\varepsilon t) \, dt \qquad \forall \varepsilon \in ]0,\varepsilon_0[\, .
\]
Then  analyticity of $u^a$ and $u^b$ (see 
{Definition \ref{d:function}}) and  analyticity results for the composition operator (cf.~B\"{o}hme and Tomi~\cite[p.~10]{BoTo73}, 
Henry~\cite[p.~29]{He82}, Valent~\cite[Thm.~5.2, p.~44]{Va88}), imply that there exists $\varepsilon_\xi \in ]0,\varepsilon_0]$ such that the map from $]-\varepsilon_\xi,\varepsilon_\xi[$ to $C^{0,\alpha}(\overline{\omega})$ which takes $\varepsilon$ to $(\partial_j u^l)(\varepsilon \cdot)_{|\overline{\omega}}$ is real analytic for $l=a,b$ and that
\[
\begin{split}
(\partial_j u^l )(\varepsilon t) &=\sum_{h=0}^\infty \partial_j u^l_{\#,h+1}(t) \varepsilon^{h} \qquad \forall t \in \overline{\omega}\, , \qquad l=a,b\, ,
\end{split}
\]
where for $l=a,b$ the series $\sum_{h=0}^\infty \partial_j u^l_{\#, h+1|\overline{\omega}} \varepsilon^h$ converges normally in $C^{0,\alpha}(\overline{\omega})$  for $\varepsilon \in ]-\varepsilon_\xi,\varepsilon_\xi[$. Accordingly,
\[
(\partial_j u^a )(\varepsilon t) (\partial_j u^b)(\varepsilon t)=  \sum_{n=0}^\infty \Bigg(\sum_{l=0}^n \partial_j u^a_{\#,l+1}(t)\partial_j u^b_{\#,n-l+1}(t)\Bigg) \varepsilon^n \qquad \forall t \in \overline{\omega}\, , \forall \varepsilon \in ]-\varepsilon_\xi,\varepsilon_\xi[\setminus \{0\}\, .
\]
{Possibly taking a smaller $\varepsilon_\xi$, we also have
\begin{equation}\label{eq:nrguom1}
\int_{ \omega}\nabla u^a(\varepsilon t) \cdot  \nabla u^b(\varepsilon t) \, dt=\sum_{n=0}^\infty \bigg(\sum_{j=1}^d\sum_{l=0}^n \int_{\omega}\partial_j u^a_{\#,l+1}(t)\partial_j u^b_{\#,n-l+1}(t)\, dt\bigg) \varepsilon^n\, ,
\end{equation}
for all $\varepsilon \in ]0,\varepsilon_\xi[$. Moreover, 
\[
\begin{split}
\sum_{j=1}^d \int_{\omega}\partial_j u^a_{\#,1}(t)\partial_j u^b_{\#,1}(t)\, dt&= \nabla u^a(0) \cdot \nabla u^a(0) m_d(\omega)\, .
\end{split}
\]
Finally, to deduce the validity of the lemma, it is enough to multiply equation \eqref{eq:nrguom1} by $\varepsilon^d$.\qed

\medskip

To deduce our main result on the asymptotic behavior of $\mathrm{Cap}_\Omega(\varepsilon \overline{\omega},u^a,u^b)$, it suffices to integrate formula \eqref{fuepsmseries} over $\partial \omega$   adding the coefficients of Lemma \ref{lem:nrguom} and to apply Theorem \ref{umk}.

\begin{theorem}\label{capk}
With the notation introduced in Proposition \ref{uk}, Theorem \ref{umk} and Lemma \ref{lem:nrguom}, let $\{c_n\}_{n\in \mathbb{N}}$ be the sequence of real numbers defined by
\[
c_n\equiv 0 \qquad \forall n \in \{0,\dots,d-3\}\, ,\qquad  c_n\equiv-\int_{\partial \omega}\tilde{\lambda}_{n-(d-2)}\, d\sigma+\xi_n \qquad \forall n \geq d-2\, .
\] 
 Then there exists $\varepsilon_\mathrm{c}\in]0,\varepsilon_0]$ such that 
\[
\mathrm{Cap}_\Omega(\varepsilon \overline{\omega},u^a,u^b)=\sum_{n=0}^\infty c_n \varepsilon^n
\]
for all $\varepsilon\in]0,\varepsilon_\mathrm{c}[$. Moreover, the series 
\[
\sum_{n=0}^\infty c_n \varepsilon^n
\] 
converges  absolutely for $\varepsilon \in]-\varepsilon_\mathrm{c},\varepsilon_\mathrm{c}[$ and
\begin{equation}\label{eq:cd-2}
\begin{split}
c_{d-2}=& -\int_{\partial \omega} \frac{u^a(0) u^b(0)}{r_0} \frac{\partial}{\partial \nu_{\omega}}v^-[\partial \omega, \rho^i_0]\, d\sigma=-\frac{u^a(0) u^b(0)}{{r_0}}\int_{\partial \omega}\rho^i_0\, d\sigma=-\frac{u^a(0) u^b(0)}{r_0}\, .
\end{split}
\end{equation}
\end{theorem}

Our next aim is to better understand the value $r_0$ which appears in formula \eqref{eq:cd-2} and, possibly, to link it to some boundary value problem related to the geometric setting. We begin with the lemma below where $v^+[\partial \omega, \rho^i_0]$ is related to the solution of some exterior Dirichlet problem in $\mathbb{R}^d\setminus \omega$.

\begin{lem}\label{rhoi0}
Let $H^i\in C^{1,\alpha}_{\mathrm{loc}}(\mathbb{R}^d\setminus\omega)$ be the solution of
\begin{equation}\label{auxbvpHi}
\left\{
\begin{array}{ll}
\Delta H^i=0&\text{in }\mathbb{R}^d\setminus\overline{\omega}\,,\\
H^i(t)=1 &\text{for all }t\in\partial\omega\,,\\
\lim_{t\to \infty}H^i(t)=0\,.
\end{array} 
\right.
\end{equation}
Then the restriction $v^+[\partial \omega, \rho^i_0]$ is constant and equal to   $\big((2-d)s_d\lim_{t\to \infty}|t|^{d-2}H^i(t)\big)^{-1}$.
\end{lem}
\proof Let $u\in C^{1,\alpha}_\mathrm{loc}(\mathbb{R}^d\setminus\omega)$, $\Delta u=0$ in $\mathbb{R}^d\setminus\overline{\omega}$, and $\lim_{t\to \infty}u(t)=0$. Then by classical potential theory there exists $\mu\in C^{1,\alpha}(\partial\omega)$ such that 
\[
u=w^-_{\omega}[\mu]+\frac{\int_{\partial \omega} u \rho^i_0\, d\sigma}{ \frac{1}{\int_{\partial\omega}d\sigma} \int_{\partial \omega}v[\partial \omega,\rho^i_0]\, d\sigma}v^-[\partial \omega,\rho^i_0]
\] 
(cf., {\it e.g.}, Folland \cite[Chap.~3]{Fo95}). Then 
\[
\begin{split}
\lim_{t\to \infty}|t|^{d-2}u(t)&=\lim_{t\to \infty}|t|^{d-2}w^-_{\omega}[\mu](t)+\lim_{t\to \infty}|t|^{d-2}\frac{\int_{\partial \omega} u \rho^i_0\, d\sigma}{\frac{1}{\int_{\partial\omega}d\sigma} \int_{\partial \omega}v[\partial \omega,\rho^i_0]\, d\sigma}v^-[\partial \omega,\rho^i_0](t)\\
&=\lim_{t\to \infty}\frac{\int_{\partial \omega} u \rho^i_0\, d\sigma}{\frac{1}{\int_{\partial\omega}d\sigma}\int_{\partial \omega}v[\partial \omega,\rho^i_0]\, d\sigma}|t|^{d-2}v^-[\partial \omega,\rho^i_0](t)\\
&=\frac{\int_{\partial \omega} u \rho^i_0\, d\sigma}{\frac{1}{\int_{\partial\omega}d\sigma}\int_{\partial \omega}v[\partial \omega,\rho^i_0]\, d\sigma}\frac{1}{(2-d)s_d}\int_{\partial \omega}\rho^i_0\, d\sigma\\
&=\frac{1}{(2-d)s_d}\frac{\int_{\partial \omega} u \rho^i_0\, d\sigma}{\frac{1}{\int_{\partial\omega}d\sigma}\int_{\partial \omega}v[\partial \omega,\rho^i_0]\, d\sigma}\, .\end{split}
\] 
As a consequence,
\[
\begin{split}
\lim_{t\to \infty}|t|^{d-2}H^i(t)&=\frac{1}{(2-d)s_d}\frac{\int_{\partial \omega} 1 \rho^i_0\, d\sigma}{\frac{1}{\int_{\partial \omega}\, d\sigma}\int_{\partial \omega}v[\partial \omega,\rho^i_0]\, d\sigma}\\
&=\frac{1}{(2-d)s_d}\frac{1}{ \frac{1}{\int_{\partial\omega}d\sigma}  \int_{\partial \omega}v[\partial \omega,\rho^i_0]\, d\sigma}\, ,
\end{split}
\] 
 and thus
 \[
 \frac{1}{\frac{1}{\int_{\partial\omega}d\sigma} \int_{\partial \omega}v[\partial \omega,\rho^i_0]\, d\sigma}=(2-d)s_d\lim_{t\to \infty}|t|^{d-2}H^i(t)\, .
 \]
Moreover, by the jump properties of the single layer potential we have $\nu_{\omega}\cdot\nabla v^+[\partial \omega, \rho^i_0]_{|\partial\omega}=0$. Thus $v^+[\partial \omega, \rho^i_0]$ is constant in $\overline{\omega}$ and the validity of the statement follows.
\qed

\begin{cor}\label{corrhoi0}
Let $g\in C^{1,\alpha}(\partial \omega)$. Let $u\in C^{1,\alpha}_{\mathrm{loc}}(\mathbb{R}^d\setminus\omega)$ be the solution of
\[
\left\{
\begin{array}{ll}
\Delta u=0&\text{in }\mathbb{R}^d\setminus\overline{\omega}\,,\\
u(t)=g(t)&\text{for all }t\in\partial\omega\,,\\
\lim_{t\to \infty}u(t)=0\,.
\end{array} 
\right.
\] 
Then 
\[
\lim_{t\to \infty}|t|^{d-2}u(t)=\frac{1}{(2-d)s_d}\frac{\int_{\partial \omega} u \rho^i_0\, d\sigma}{\frac{1}{\int_{\partial\omega}d\sigma}\int_{\partial \omega}v[\partial \omega,\rho^i_0]\, d\sigma}=\int_{\partial \omega} g \rho^i_0\, d\sigma  \lim_{t\to \infty}|t|^{d-2}H^i(t)\, .
\]
\end{cor}

\begin{rem}\label{rem:capepsfirst}
Let $H^i$ be the unique solution  in $C^{1,\alpha}_{\mathrm{loc}}(\mathbb{R}^d\setminus\omega)$ of problem \eqref{auxbvpHi}. Then by Lemma \ref{rhoi0} we have
\[
-\frac{1}{r_0}= (d-2)s_d\lim_{t\to \infty}|t|^{d-2}H^i(t)\, .
\]
Accordingly,
\begin{equation}\label{eq:capepsfirst}
\begin{split}
\mathrm{Cap}_\Omega(\varepsilon \overline{\omega},u^a,u^b)=& u^a(0) u^b(0) (d-2)s_d\lim_{t\to \infty}|t|^{d-2}H^i(t)\varepsilon^{d-2}+\varepsilon^{d-1}\bigg(\sum_{n=d-1}^\infty c_n \varepsilon^{n-(d-1)}\bigg)
\end{split}
\end{equation}
for all $\varepsilon\in]0,\varepsilon_\mathrm{c}[$. 

We now wish to provide an alternative characterization of the quantity
\[
(d-2)s_d\lim_{t\to \infty}|t|^{d-2}H^i(t)\, .
\] 
By the Divergence theorem in exterior domains {for functions which are harmonic at infinity (see Dalla Riva, Lanza de Cristoforis, and Musolino \cite[\S 4.2]{DaLaMu21})}, we have
\[
\begin{split}
\int_{\mathbb{R}^d\setminus \overline{\omega}}|\nabla H^i(t)|^2\, dt&=-\int_{\partial \omega}H^i(t)\frac{\partial}{\partial \nu_{\omega}}H^i(t)\, d\sigma_t\\
&=-\int_{\partial \omega}\frac{\partial}{\partial \nu_{\omega}}H^i(t)\, d\sigma_t\, .
\end{split}
\]
On the other hand, {by Lemma \ref{rhoi0}} one verifies that
\[
H^i=\big((2-d)s_d\lim_{t\to \infty}|t|^{d-2}H^i(t)\big) v^-[\partial \omega, \rho^i_0] \qquad \text{in } \mathbb{R}^d \setminus \omega\, ,
\]
and that
\begin{equation}\label{eq:Hder}
\begin{split}
\frac{\partial}{\partial \nu_{\omega}}H^i&=\big((2-d)s_d\lim_{t\to \infty}|t|^{d-2}H^i(t)\big)\Bigg(\frac{1}{2}\rho^i_0+W^\ast_\omega[\rho^i_0]\Bigg)\\
&=\big((2-d)s_d\lim_{t\to \infty}|t|^{d-2}H^i(t)\big) \rho^i_0 \qquad \text{on } \partial \omega\, .
\end{split}
\end{equation}
As a consequence,
\[
\begin{split}
-\int_{\partial \omega}\frac{\partial}{\partial \nu_{\omega}}H^i(t)\, d\sigma_t&=\big((d-2)s_d\lim_{t\to \infty}|t|^{d-2}H^i(t)\big) \int_{\partial \omega}\rho^i_0 \, d\sigma\\
&=\big((d-2)s_d\lim_{t\to \infty}|t|^{d-2}H^i(t)\big)\, .
\end{split}
\]
Accordingly,
\[
\int_{\mathbb{R}^d\setminus \overline{\omega}}|\nabla H^i(t)|^2\, dt=(d-2)s_d\lim_{t\to \infty}|t|^{d-2}H^i(t)\, .
\]
In other words, the quantity
\[
(d-2)s_d\lim_{t\to \infty}|t|^{d-2}H^i(t)
\]
equals the energy integral of $H^i$, {\it i.e.}, of the unique function in $C^{1,\alpha}_{\mathrm{loc}}(\mathbb{R}^d\setminus \omega)$ which is harmonic in $\mathbb{R}^d \setminus \overline{\omega}$ and at infinity and that is equal to $1$ on $\partial \omega$. {Such quantity is thus equal to the Newtonian capacity $\mathrm{Cap}_{\mathbb{R}^d}(\overline{\omega})$ (see \eqref{eq:Newcap}).}
\end{rem}

{\begin{rem}
In case $\omega$ is equal to the open unit ball $B_1$ in $\mathbb{R}^d$ of center $0$ and radius $1$, one verifies that
\[
H^i(t)=|t|^{2-d}\qquad \forall t \in \mathbb{R}^d\setminus B_1\, .
\]
As a consequence,
\[
(d-2)s_d\lim_{t\to \infty}|t|^{d-2}H^i(t)=(d-2)s_d\, ,
\]
and thus
\[
\mathrm{Cap}_{\mathbb{R}^d}(\overline{B_1})=(d-2)s_d\, .
\]
In particular, if $d=3$, then 
\[
\mathrm{Cap}_{\mathbb{R}^3}(\overline{B_1})=4\pi
\]
\end{rem}}

\section{Asymptotic behavior of $\mathrm{Cap}_\Omega(\varepsilon \overline{\omega},u^a,u^b)$ under vanishing assumption for $u^a$ and $u^b$}\label{s:vanish}

In this section,
we investigate the behavior of $\mathrm{Cap}_\Omega(\varepsilon \overline{\omega},u^a,u^b)$ assuming that the functions $u^a$ and $u^b$ and their derivatives up to a certain order {could} vanish at $0$ and we modify the computation of \cite[\S 5.1]{AbBoLeMu21} and \cite[\S 2.7]{AbLeMu22}. {So we consider the following assumption  
\begin{equation}\label{eq:vanu}
\begin{split}
&\mbox{$u^a$ and $u^b$ are admissible functions as in Definition \ref{d:function},} \\
&\mbox{not identically zero in a neighborhood of $0$.}
\end{split}
\end{equation}
Assumption \eqref{eq:vanu} implies that {$u^a,u^b$ have finite order of vanishing $\overline{k}^a, \overline{k}^b \in \mathbb{N}$ at $0$, that is}
\[
\begin{split}
&D^\gamma u^l(0)=0 \quad \forall |\gamma| <\overline{k}^l\, , \qquad D^{\beta^l} u^l(0)\neq 0 \quad \mbox{for some $\beta^l \in \mathbb{N}^d$ with $|\beta^l|=\overline{k}^l$}\, ,\qquad l=a,b\,.
\end{split}
\]} 
By Remark \ref{rem:capepsfirst}, we already know   the principal term of the asymptotic expansion of $\mathrm{Cap}_\Omega(\varepsilon \overline{\omega},u^a,u^b)$ as $\varepsilon \to 0$ in case $\overline{k}^a= \overline{k}^b =0$ ({\it i.e.}, $u^a(0) u^b(0)\neq 0$). We now wish to investigate the case when at least one between $u^a$ and $u^b$ vanishes at $0$. The computation below notably simplifies if $\overline{k}^a$ or $\overline{k}^b$ is equal to $0$, although it is not necessary to assume that this is the case. 
} By condition \eqref{eq:vanu} and Proposition \ref{thetak} we note that
\begin{equation}\label{eq:vantheta}
(\theta^o_k,\theta^i_k)=(0,0) \quad \forall k < \overline{k}{^a}\, , \qquad \theta^o_{\overline{k}^a}=0\, ,
\end{equation}
and that $\theta^i_{\overline{k}^a}$ is the unique solution in $C^{1,\alpha}(\partial\omega)_0$ of  
\[
\begin{split}
\frac{1}{2}\theta^i_{\overline{k}^a}(t)-W_{\omega}[\theta^i_{\overline{k}^a}](t)&=\sum_{|\beta|=\overline{k}^a}\frac{\overline{k}^a!}{\beta!}t^\beta (D^{\beta}u^a)(0)- \sum_{\substack{\beta \in \mathbb{N}^d\\ |\beta|=\overline{k}^a}} \frac{\overline{k}^a!}{\beta!}\int_{\partial\omega}{s^{\beta}(D^\beta u^a)(0)\rho^i_{0}(s) \, d\sigma_s}\qquad \forall t\in\partial\omega\,,
\end{split}
\]
\textit{i.e.}, 
\begin{equation}\label{eq:thetaooverk}
\begin{split}
\frac{1}{2}\theta^i_{\overline{k}^a}(t)-W_{\omega}[\theta^i_{\overline{k}^a}](t)&=\overline{k}^a! \Bigg (u_{\#,\overline{k}^a}(t)-\int_{\partial\omega}u_{\#,\overline{k}^a}\rho^i_{0}\,d\sigma\Bigg) \qquad \forall t\in\partial\omega\,.
\end{split}
\end{equation}
Then equations \eqref{eq:vantheta}, \eqref{eq:thetaooverk}, and Proposition \ref{uk} imply that
\begin{equation}\label{eq:vanum}
u^a_{\mathrm{m},k}=0 \qquad \forall k <\overline{k}^a\, ,\qquad u^a_{\mathrm{m},\overline{k}^a}=-\frac{1}{\overline{k}^a!}w^-[\partial \omega, \theta^i_{\overline{k}^a}]\, .
\end{equation}
Hence, by the properties of the double layer potential, one verifies that $u^a_{\mathrm{m},\overline{k}^a}$ is the unique solution in $C^{1,\alpha}_{\mathrm{loc}}(\mathbb{R}^d \setminus \omega)$ of 
\[
\left\{
\begin{array}{ll}
\Delta u^a_{\mathrm{m},\overline{k}^a}=0&\text{in }\mathbb{R}^d\setminus\overline{\omega}\,,\\
u^a_{\mathrm{m},\overline{k}^a}(t)=u^a_{\#,\overline{k}^a}(t)-\int_{\partial\omega}u^a_{\#,\overline{k}^a}\rho^i_{0}\,d\sigma&\text{for all }t\in\partial\omega\,,\\
\lim_{t\to \infty}u^a_{\mathrm{m},\overline{k}^a}(t)=0\,.
\end{array} 
\right.
\]
In addition,  by assumption \eqref{eq:vanu} and Proposition \ref{uk} we deduce that
\begin{equation}\label{eq:vang}
g^a_{k}=0 \quad \forall k <\overline{k}^a\, ,\qquad g^a_{\overline{k}^a}=\int_{\partial\omega}u^a_{\#,\overline{k}^a}\rho^i_{0}\,d\sigma\, .
\end{equation}
Then by \eqref{eq:vanu} and by Propostion \ref{uk} we show that
\begin{equation}\label{eq:vanusharp}
u^l_{\#,k}=0 \qquad \forall k <\overline{k}^l\, , \qquad l=a,b\, .
\end{equation}
As a consequence, Proposition \ref{uk} and equations \eqref{eq:vanum}, \eqref{eq:vanusharp} imply
\begin{equation}\label{eq:vantildeu}
\tilde{u}_{k}=0 \qquad \forall k <\overline{k}^a+\overline{k}^b\, ,\qquad \tilde{u}_{\overline{k}^a+\overline{k}^b}=\Bigg( \frac{\partial u^a_{\mathrm{m},\overline{k}^a}}{\partial \nu_\omega}\Bigg)u^b_{\#,\overline{k}^b|\partial \omega}\, .
\end{equation}
By \eqref{eq:vang} and \eqref{eq:vanusharp} we have
\begin{equation}\label{eq:vantildeg}
\tilde{g}_{k}=0 \qquad \forall k <\overline{k}^a+\overline{k}^b\, ,\qquad \tilde{g}_{\overline{k}^a+\overline{k}^b}=g^a_{\overline{k}^a}u^b_{\#,\overline{k}^b|\partial \omega}=
 \bigg(\int_{\partial\omega}u^a_{\#,\overline{k}^a}\rho^i_{0}\,d\sigma \bigg)  u^b_{\#,\overline{k}^b|\partial \omega} \, .
\end{equation}

We now consider the quantities $\tilde{a}_{n}, \tilde{\lambda}_{n}$ introduced in Theorem \ref{umk} for representing the behavior of $\nu_{\omega}(\cdot) \cdot \nabla \big(u^a_{\varepsilon}(\varepsilon\cdot)\big) u^b(\varepsilon \cdot)$. By a direct computation based on \eqref{eq:vantildeu}, \eqref{eq:vantildeg} we have
\[
\tilde{a}_{n}=0 \qquad \forall n <\overline{k}^a+\overline{k}^b\, ,\qquad \tilde{a}_{\overline{k}^a+\overline{k}^b}=\tilde{g}_{\overline{k}^a+\overline{k}^b}\tilde{v}_0=  \tilde{v}_0
\Big(\int_{\partial\omega}u^a_{\#,\overline{k}^a}\rho^i_{0}\,d\sigma \Big) u^b_{\#,\overline{k}^b|\partial \omega} \, ,
 \]
and thus
\begin{equation}\label{eq:vantilde0}
\begin{split}
&\tilde{\lambda}_{n}=0 \qquad \forall n <\overline{k}^a+\overline{k}^b\, ,\\
& \tilde{\lambda}_{\overline{k}^a+\overline{k}^b}=\tilde{u}_{\overline{k}^a+\overline{k}^b}+\frac{\tilde{a}_{\overline{k}^a+\overline{k}^b}}{r_0}= \Bigg( \frac{\partial u^a_{\mathrm{m},\overline{k}^a}}{\partial \nu_\omega}\Bigg)u^b_{\#,\overline{k}^b|\partial \omega}+\frac{1}{r_0} \tilde{v}_0
\Big(\int_{\partial\omega}u^a_{\#,\overline{k}^a}\rho^i_{0}\,d\sigma \Big) u^b_{\#,\overline{k}^b|\partial \omega} \, .
\end{split}
\end{equation}
Moreover, a simple computation shows that
\[
\xi_{n}=0 \qquad \forall n < \overline{k}^a+\overline{k}^b+d-2\, , \qquad \xi_{\overline{k}^a+\overline{k}^b+d-2}= \int_{\omega} \nabla u^a_{\#,\overline{k}^a}\cdot \nabla u^b_{\#,\overline{k}^b} \, dt\, .
\]
Finally, by Theorem \ref{capk} and by integrating equalities \eqref{eq:vantilde0}, we obtain 
\[
\begin{split}
&c_{n}=0 \qquad \forall n <\overline{k}^a+\overline{k}^b+d-2\, ,\\ 
&c_{\overline{k}^a+\overline{k}^b+d-2}=-\int_{\partial \omega}\Big(\tilde{u}_{\overline{k}^a+\overline{k}^b}+\frac{\tilde{a}_{\overline{k}^a+\overline{k}^b}}{r_0}\Big)\, d\sigma+\int_{\omega} \nabla u^a_{\#,\overline{k}^a}\cdot \nabla u^b_{\#,\overline{k}^b} \, dt\\
&\qquad=-\int_{\partial \omega} \Bigg( \frac{\partial u^a_{\mathrm{m},\overline{k}^a}}{\partial \nu_\omega}\Bigg)u^b_{\#,\overline{k}^b|\partial \omega}\, d\sigma-\frac{1}{r_0} \int_{\partial \omega}\tilde{v}_0
u^b_{\#,\overline{k}^b}\, d\sigma \int_{\partial\omega}u^a_{\#,\overline{k}^a}\rho^i_{0}\,d\sigma+\int_{\omega} \nabla u^a_{\#,\overline{k}^a}\cdot \nabla u^b_{\#,\overline{k}^b} \, dt\,\\
&\qquad=-\int_{\partial \omega} \Bigg( \frac{\partial u^a_{\mathrm{m},\overline{k}^a}}{\partial \nu_\omega}\Bigg)u^b_{\#,\overline{k}^b|\partial \omega}\, d\sigma-\frac{1}{r_0} \int_{\partial\omega}u^a_{\#,\overline{k}^a}\rho^i_{0}\,d\sigma \int_{\partial\omega}u^b_{\#,\overline{k}^b}\rho^i_{0}\,d\sigma+\int_{\omega} \nabla u^a_{\#,\overline{k}^a}\cdot \nabla u^b_{\#,\overline{k}^b} \, dt\,.
\end{split}
\]
{
Now let $H_{u^a,\overline{k}^a}\in C^{1,\alpha}_{\mathrm{loc}}(\mathbb{R}^d\setminus\omega)$ be the solution of
\[
\left\{
\begin{array}{ll}
\Delta H_{u^a,\overline{k}^a}=0&\text{in }\mathbb{R}^d\setminus\overline{\omega}\,,\\
H_{u^a,\overline{k}^a}(t)= \int_{\partial\omega}u^a_{\#,\overline{k}^a}\rho^i_{0}\,d\sigma &\text{for all }t\in\partial\omega\,,\\
\lim_{t\to \infty}H_{u^a,\overline{k}^a}=0\,.
\end{array} 
\right.
\]
Then by Remark \ref{rem:capepsfirst} we have
\[
H_{u^a,\overline{k}^a}=\Bigg(\int_{\partial\omega}u^a_{\#,\overline{k}^a}\rho^i_{0}\,d\sigma\Bigg)H^i \qquad \text{in $\mathbb{R}^d\setminus \omega$}\, ,
\]
and accordingly by \eqref{eq:Hder}
\[
\frac{\partial}{\partial \nu_{\omega}}H_{u^a,\overline{k}^a}=\frac{1}{r_0}\Bigg(\int_{\partial\omega}u^a_{\#,\overline{k}^a}\rho^i_{0}\,d\sigma\Bigg)\rho^i_{0}\qquad \text{on $\partial \omega$}\, .
\]
As a consequence,
\[
\frac{1}{r_0} \int_{\partial\omega}u^a_{\#,\overline{k}^a}\rho^i_{0}\,d\sigma \int_{\partial\omega}u^b_{\#,\overline{k}^b}\rho^i_{0}\,d\sigma=\int_{\partial \omega}\Bigg(\frac{\partial H_{u^a,\overline{k}^a}}{\partial \nu_{\omega}}\Bigg)u^b_{\#,\overline{k}^b|\partial \omega}\, d\sigma\, .
\]
Then for $l=a,b$ we denote by $\mathsf{u}^l_{\overline{k}^l}$ the unique solution in $C^{1,\alpha}_{\mathrm{loc}}(\mathbb{R}^d\setminus \omega)$ of
\begin{equation}\label{eq:bvp:mathsfu}
\left\{
\begin{array}{ll}
\Delta \mathsf{u}^l_{\overline{k}^l}=0&\text{in }\mathbb{R}^d\setminus\overline{\omega}\,,\\
\mathsf{u}^l_{\overline{k}^l}(t)=u^l_{\#,\overline{k}^l}(t)&\text{for all }t\in\partial\omega\,,\\
\lim_{t\to \infty}\mathsf{u}^l_{\overline{k}^l}(t)=0\,.
\end{array} 
\right.
\end{equation}
Thus
\[
\mathsf{u}^a_{\overline{k}^a}=u^a_{\mathrm{m},\overline{k}^a}+H_{u^a,\overline{k}^a}\qquad \text{in $\mathbb{R}^d \setminus \omega$}\, ,
\]
and therefore
\[
-\int_{\partial \omega} \Bigg( \frac{\partial u^a_{\mathrm{m},\overline{k}^a}}{\partial \nu_\omega}\Bigg)u^b_{\#,\overline{k}^b|\partial \omega}\, d\sigma-\frac{1}{r_0} \int_{\partial\omega}u^a_{\#,\overline{k}^a}\rho^i_{0}\,d\sigma \int_{\partial\omega}u^b_{\#,\overline{k}^b}\rho^i_{0}\,d\sigma=-\int_{\partial \omega}\Bigg(\frac{\partial \mathsf{u}^a_{\overline{k}^a}}{\partial \nu_\omega}\Bigg)\mathsf{u}^b_{\overline{k}^b}\, d\sigma\, .
\]
On the other hand,  the harmonicity at infinity of $\mathsf{u}^a_{\overline{k}^a}$ and of  $\mathsf{u}^b_{\overline{k}^b}$ and the Divergence Theorem imply that
\[
{\int_{\mathbb{R}^d \setminus \overline{\omega}}\nabla \mathsf{u}^a_{\overline{k}^a} \cdot \nabla \mathsf{u}^b_{\overline{k}^b}}\, dt=-\int_{\partial \omega}\Bigg(\frac{\partial \mathsf{u}^a_{\overline{k}^a}}{\partial \nu_\omega}\Bigg)\mathsf{u}^b_{\overline{k}^b}\, d\sigma\, 
\]
(cf.~Folland \cite[p.~118]{Fo95}, Dalla Riva, Lanza de Cristoforis, and Musolino \cite[Cor.~4.7]{DaLaMu21}). 
Accordingly,
\[
\begin{split}
c_{\overline{k}^a+\overline{k}^b+d-2}=\int_{\mathbb{R}^d \setminus \overline{\omega}}\nabla \mathsf{u}^a_{\overline{k}^a} \cdot \nabla \mathsf{u}^b_{\overline{k}^b}\, dt+\int_{\omega} \nabla u^a_{\#,\overline{k}^a}\cdot \nabla u^b_{\#,\overline{k}^b} \, dt\, .
\end{split}
\]
{Incidentally, we note that if for example $u^b(0)\neq 0$ ({\it i.e.}, if $\overline{k}^b=0$), then $u^b_{\#,\overline{k}^b}=u^b(0)$ (and so $\nabla u^b_{\#,\overline{k}^b}=0$) and $\mathsf{u}^b_{\overline{k}^b}=u^b(0)H^i$. Therefore, if $\overline{k}^b=0$ the term $c_{\overline{k}^a+\overline{k}^b+d-2}$ reduces to ${u}^b(0)\int_{\mathbb{R}^d \setminus \overline{\omega}}\nabla \mathsf{u}^a_{\overline{k}^a} \cdot \nabla H^i\, dt$. Similarly, if $\overline{k}^a=0$ the term $c_{\overline{k}^a+\overline{k}^b+d-2}$ reduces to ${u}^a(0)\int_{\mathbb{R}^d \setminus \overline{\omega}}\nabla \mathsf{u}^b_{\overline{k}^b} \cdot \nabla H^i\, dt$. We also note that if both $\overline{k}^a=0$ and $\overline{k}^b=0$, then for $l=a,b$ we have $u^l_{\#,\overline{k}^l}=u^l(0)$ (and so $\nabla u^l_{\#,\overline{k}^l}=0$) and $\mathsf{u}^l_{\overline{k}^l}=u^l(0)H^i$, and accordingly
\[
\begin{split}
\int_{\mathbb{R}^d \setminus \overline{\omega}}\nabla \mathsf{u}^a_{\overline{k}^a} \cdot \nabla \mathsf{u}^b_{\overline{k}^b}\, dt+\int_{\omega} \nabla u^a_{\#,\overline{k}^a}\cdot \nabla u^b_{\#,\overline{k}^b} \, dt&=u^a(0)u^b(0)\int_{\mathbb{R}^d \setminus \overline{\omega}}|\nabla H^i|^2\, dt\\
&=u^a(0)u^b(0)(d-2)s_d\lim_{t\to \infty}|t|^{d-2}H^i(t)\, .
\end{split}
\]
In other words, the quantity $u^a(0)u^b(0)(d-2)s_d\lim_{t\to \infty}|t|^{d-2}H^i(t)$ can be seen as the specific value of $\int_{\mathbb{R}^d \setminus \overline{\omega}}\nabla \mathsf{u}^a_{\overline{k}^a} \cdot \nabla \mathsf{u}^b_{\overline{k}^b}\, dt+\int_{\omega} \nabla u^a_{\#,\overline{k}^a}\cdot \nabla u^b_{\#,\overline{k}^b} \, dt$ when both $\overline{k}^a$ and $\overline{k}^b$ are equal to $0$.}

As a consequence, under assumption \eqref{eq:vanu}, by Theorem \ref{capk} and formula \eqref{eq:capepsfirst}, we can deduce the validity of the following.

 \begin{theorem}\label{thm:cepsmseries}
Let assumption \eqref{eq:vanu} hold. For  $l=a,b$, let $\mathsf{u}^l_{\overline{k}^l}$ be the unique solution in $C^{1,\alpha}_{\mathrm{loc}}(\mathbb{R}^d\setminus \omega)$ of \eqref{eq:bvp:mathsfu}.Then
 \begin{equation}\label{cepsmseries}
\begin{split}
\mathrm{Cap}_{\Omega}&(\varepsilon \overline{\omega},u^a,u^b)\\
=&\varepsilon^{\overline{k}^a+\overline{k}^b+d-2}\Bigg(\int_{\mathbb{R}^d \setminus \overline{\omega}}\nabla \mathsf{u}^a_{\overline{k}^a} \cdot \nabla \mathsf{u}^b_{\overline{k}^b}\, dt+\int_{\omega} \nabla u^a_{\#,\overline{k}^a}\cdot \nabla u^b_{\#,\overline{k}^b} \, dt\Bigg)+\sum_{n=\overline{k}^a+\overline{k}^b+d-1}^\infty\varepsilon^n c_{n}\, ,
\end{split}
\end{equation} 
for all $\varepsilon\in]0,\varepsilon_\mathrm{c}[$. 
\end{theorem}

\begin{rem}\label{rem:cepsmseries}
Under assumption \eqref{eq:vanu}, by \eqref{cepsmseries} we have
 \[
\begin{split}
\mathrm{Cap}_{\Omega}&(\varepsilon \overline{\omega},u^a,u^b)\\
=&\varepsilon^{\overline{k}^a+\overline{k}^b+d-2}\Bigg(\int_{\mathbb{R}^d \setminus \overline{\omega}}\nabla \mathsf{u}^a_{\overline{k}^a} \cdot \nabla \mathsf{u}^b_{\overline{k}^b}\, dt+\int_{\omega} \nabla u^a_{\#,\overline{k}^a}\cdot \nabla u^b_{\#,\overline{k}^b} \, dt\Bigg) +o(\varepsilon^{\overline{k}^a+\overline{k}^b+d-2})\, 
 \qquad \text{as }\varepsilon \to 0\, .
\end{split}
\]
Moreover, we note that the coefficient of $\varepsilon^{\overline{k}^a+\overline{k}^b+d-2}$ depends  both on the geometrical properties of the set $\omega$ and on the behavior at $0$ of the functions $u^a$ and $u^b$, but does not depend on $\Omega$.
\end{rem}
}

\section{Asymptotic behavior of the eigenvalues of the Dirichlet-Laplacian in perforated domains} \label{sec6}

It is well known that if $\Omega$ is a bounded open set in $\mathbb{R}^d$, $K$ a compact subset of $\Omega$, and {if we denote} by
\[
0<\lambda_1(\Omega)<\lambda_2(\Omega)\leq \dots \leq \lambda_N(\Omega)\leq \dots
\]
and
\[
0<\lambda_1(\Omega \setminus K)<\lambda_2(\Omega \setminus K)\leq \dots \leq \lambda_N(\Omega \setminus K)\leq \dots
\]
the sequences of the eigenvalues of the Dirichlet-Laplacian in $\Omega$ and in $\Omega \setminus K$, respectively, then $\lambda_N(\Omega\setminus K)$ is close to $\lambda_N(\Omega)$ if and only if the capacity $\mathrm{Cap}_\Omega(K)$ of $K$ in $\Omega$ is small (see Rauch and Taylor \cite{RaTa75}).  A typical example is when {we fix} $\Omega$ and $\omega$ admissible domains, and we set
\[
K=\varepsilon \overline{\omega} \qquad\quad\forall\varepsilon\in]-\varepsilon_\#,\varepsilon_\#[\, ,
\]
with $\varepsilon_\#$ as in \eqref{eq:varepssharp}. 
{We define}
\[
\Omega_\varepsilon\equiv\Omega\setminus (\eps\overline\omega)
\qquad\quad\forall\varepsilon\in]-\varepsilon_\#,\varepsilon_\#[\,,
\]
{as before,} and we wish to study the convergence of the $N$-th Dirichlet eigenvalues $\lambda_N(\Omega_\varepsilon)=\lambda_N(\Omega \setminus (\varepsilon \overline{\omega}))$ to $\lambda_N(\Omega)$ as $\varepsilon \to 0$.

In this section we show how the results on the asymptotic behavior of the generalizations of the capacity can be employed to obtain accurate asymptotic {estimates}
for  $\lambda_N(\Omega_\varepsilon)$ when $\varepsilon \to 0$, both in the case {when}
 $\lambda_N(\Omega)$ is a simple eigenvalue of the Dirichlet-Laplacian in $\Omega$ and {in the case when it is a multiple eigenvalue,} as we have done in \cite{AbBoLeMu21, AbLeMu22} for the planar case. We treat these two situations in two different subsections.

\subsection{Simple eigenvalues}\label{subsec:simple}

We begin our analysis by considering the case when  $\lambda_N(\Omega)$ is a simple eigenvalue of the Dirichlet-Laplacian in the open set $\Omega$

The following result by  Courtois \cite[Proof of Theorem 1.2]{Co95} and Abatangelo, Felli, Hillairet, and L\'ena \cite[Theorem 1.4]{AbFeHiLe19} shows that the $u$-capacity can be succesfully used to study the asymptotic behavior of the eigenvalues $\lambda_N(\Omega \setminus (\varepsilon\overline{\omega}))$ as $\varepsilon \to 0$. 

\begin{theorem}\label{thm:asy:eig} 
Let $\lambda_N(\Omega)$ be a simple eigenvalue of the Dirichlet-Laplacian in a bounded, connected, and open set $\Omega$. Let $u_N$ be a $L^2(\Omega)$-normalized eigenfunction associated to $\lambda_N(\Omega)$ and let $(K_\varepsilon)_{\varepsilon>0}$ be a family of compact sets contained in $\Omega$ concentrating to a compact set $K$ with $\mathrm{Cap}_\Omega(K) = 0$. Then
\begin{equation}\label{eq:asy:eig}
\lambda_N(\Omega \setminus K_\varepsilon)=\lambda_N(\Omega)+\mathrm{Cap}_\Omega(K_\varepsilon,u_N)+o(\mathrm{Cap}_\Omega(K_\varepsilon,u_N))\, , \qquad \text{as $\varepsilon \to 0$}\, .
\end{equation}
\end{theorem}

In view of Theorem \ref{thm:asy:eig}, we can produce an asymptotic expansion of $\lambda_N(\Omega\setminus (\varepsilon \overline{\omega}))$ by combining the expansion of ${\mathrm{Cap}_{\Omega}(\varepsilon \overline{\omega}},u)$ and the asymptotic fomula  \eqref{eq:asy:eig} for the eigenvalues.

Therefore, we fix $\Omega$ and $\omega$ admissible domains and we assume that
\begin{equation}\label{ass:eig}
\begin{split}
&\text{the $N$-th eigenvalue $\lambda_N(\Omega)$ for the Dirichlet-Laplacian  is simple}\\
&\text{and $u_N$ is a $L^2(\Omega)$-normalized eigenfunction related to $\lambda_N(\Omega)$.}
\end{split}
\end{equation} 

As a first step, we formulate our result on the asymptotic behavior of ${\mathrm{Cap}_{\Omega}(\varepsilon \overline{\omega}},u_N)$. Standard elliptic regularity theory  (see for instance Friedman \cite[Thm.~1.2, p.~205]{Fr69}) implies that $u_N$ is analytic in a neighborhood of $0$. Accordingly,  by \eqref{eq:capepsfirst} we have
\[
\begin{split}
\mathrm{Cap}_\Omega(\varepsilon \overline{\omega},u_N)=&\mathrm{Cap}_\Omega(\varepsilon \overline{\omega},u_N,u_N)\\
=&(u_N(0))^2
{\mathrm{Cap}_{\mathbb{R}^d}(\overline\omega)}\varepsilon^{d-2}+\varepsilon^{d-1}\bigg(\sum_{n=d-1}^\infty c_n \varepsilon^{n-(d-1)}\bigg)
\end{split}
\]
for $\varepsilon$ positive and small, where the sequence of coefficients $\{c_{n}\}_{n\in\mathbb{N}}$ is as in Theorem \ref{capk} and 
{$\mathrm{Cap}_{\mathbb{R}^d}(\overline\omega)$ is as in Equation \eqref{eq:Newcap}.}

Then by formula  \eqref{eq:asy:eig} we immediately deduce the validity of the following result.

\begin{theorem}\label{thm:eig1}
Let assumption \eqref{ass:eig} hold.Then
\begin{equation}\label{eq:eign1:1}
\begin{split}
\lambda_N(\Omega&\setminus (\varepsilon \overline{\omega}))\\&=\lambda_N(\Omega)+(u_N(0))^2 {\mathrm{Cap}_{\mathbb{R}^d}(\overline{\omega})}\varepsilon^{d-2}+o\Big(\varepsilon^{d-2}\Big)\qquad \text{as $\varepsilon \to 0^+$}\, {.}
\end{split}
\end{equation}
\end{theorem}
The above result agrees with the one of Maz'ya, Nazarov, and Plamenevski\u\i  \ \cite{MaNaPl84} for $N=1$ and $d=3$, {with that of Ozawa \cite{Oz81b} for $d=2,3$, and with that of Flucher \cite{Fl94} if $d\geq2$}. 

{If we assume} that 
\begin{equation}\label{eq:van:uN}
u_N(0)=0\, ,
\end{equation}
then formula \eqref{eq:eign1:1} of Theorem \ref{thm:eig1}  reduces to
\[
\begin{split}
\lambda_N(\Omega\setminus (\varepsilon \overline{\omega}))=\lambda_N(\Omega)+o\Big( \varepsilon^{d-2}\Big)\qquad \text{as $\varepsilon \to 0^+$}\, .
\end{split}
\]
Therefore, if \eqref{eq:van:uN} {holds,} 
\begin{equation}\label{eq:vanuN_0}
\begin{split}
&\text{there exists $\overline{k} \in \mathbb{N}\setminus \{0\}$ such that}\  D^\gamma u_N(0)=0 \qquad \forall |\gamma| <\overline{k}\\
&\text{and that $D^\beta u_N(0)\neq 0$ for some $\beta \in \mathbb{N}^d$ with $|\beta|=\overline{k}$}\, .
\end{split}
\end{equation}
Then we set
\begin{equation}\label{eq:uNsharp_0}
u_{N,\#,\overline{k}}(t)\equiv\sum_{\substack{\beta\in \mathbb{N}^d\\ |\beta|=\overline{k}}}\frac{D^\beta u_N(0)}{\beta !}t^\beta \qquad\qquad\quad  \forall t\in  \mathbb{R}^d\, ,
\end{equation}
and we denote by $\mathsf{u}_{N,\overline{k}}$ the unique solution in $C^{1,\alpha}_{\mathrm{loc}}(\mathbb{R}^d\setminus \omega)$ of
\begin{equation}\label{eq:uNbf_0}
\left\{
\begin{array}{ll}
\Delta \mathsf{u}_{N,\overline{k}}=0&\text{in }\mathbb{R}^d\setminus\overline{\omega}\,,\\
\mathsf{u}_{N,\overline{k}}(t)=u_{N,\#,\overline{k}}(t)&\text{for all }t\in\partial\omega\,,\\
 \lim_{t\to \infty}\mathsf{u}_{N,\overline{k}}(t)=0\,.
\end{array} 
\right.
\end{equation}
Hence, Remark \ref{rem:cepsmseries} implies that{
\[
\begin{split}
\mathrm{Cap}_{\Omega}&(\varepsilon \overline{\omega},u_N)=\mathrm{Cap}_{\Omega}(\varepsilon \overline{\omega},u_N,u_N)\\
=&\varepsilon^{2\overline{k}+d-2}\Bigg(\int_{\mathbb{R}^d \setminus \overline{\omega}}|\nabla \mathsf{u}_{N,\overline{k}}|^2\, dt+\int_{\omega} |\nabla u_{N,\#,\overline{k}}|^2 \, dt\Bigg) +o(\varepsilon^{2\overline{k}+d-2})\, \\
 =&\varepsilon^{2\overline{k}+d-2}\mathfrak{C}(\omega, u_{N,\#,\overline{k}})+o(\varepsilon^{2\overline{k}+d-2})\, 
 \qquad \text{as }\varepsilon \to 0\, 
\end{split}\]
(see \eqref{eq:frakC}).} Thus,  by formula  \eqref{eq:asy:eig} of Theorem \ref{thm:asy:eig} we  deduce the  following  result.
\begin{theorem}\label{thm:eig2}
Let {assumption \eqref{ass:eig}}
hold. Let {$\overline{k}$ be as in \eqref{eq:vanuN_0} and} $u_{N,\#,\overline{k}}$ be as in \eqref{eq:uNsharp_0}. Let  $\mathsf{u}_{N,\overline{k}}$ be the unique solution in $C^{1,\alpha}_{\mathrm{loc}}(\mathbb{R}^d\setminus \omega)$ of \eqref{eq:uNbf_0}. Then
\begin{equation}\label{eq:eign2:1}
\begin{split}
\lambda_N(\Omega&\setminus (\varepsilon \overline{\omega}))\\&=\lambda_N(\Omega)+\varepsilon^{2\overline{k}+d-2}\Bigg(\int_{\mathbb{R}^d \setminus \overline{\omega}}|\nabla \mathsf{u}_{N,\overline{k}}|^2\, dt
+\int_{\omega} |\nabla u_{N,\#,\overline{k}}|^2 \, dt\Bigg) +o(\varepsilon^{2\overline{k}+d-2})\, 
 \qquad \text{as }\varepsilon \to 0\, .
\end{split}
\end{equation}
\end{theorem}
{Clearly, Theorems \ref{thm:eig1} and \ref{thm:eig2} can be restated as Theorem \ref{t:1} of the Introduction.}
\begin{rem}\label{rem:k0}
By the computation of Section \ref{s:vanish}, we note that 
{\[
\begin{split}
\int_{\mathbb{R}^d \setminus \overline{\omega}}|\nabla \mathsf{u}_{N,0}|^2\, dt
+\int_{\omega} |\nabla u_{N,\#,0}|^2 \, dt&=(u_N(0))^2 (d-2)s_d\lim_{t\to \infty}|t|^{d-2}H^i(t)\\
&=(u_N(0))^2\mathrm{Cap}_{\mathbb{R}^d}(\overline{\omega})\, ,
\end{split}\]}
and thus equation \eqref{eq:eign2:1} holds also when the {order of vanishing} $\overline{k}$ of $u_N$ is equal to $0$. 
\end{rem}

\subsection{Multiple eigenvalues}\label{subsec:multiple}

In the previous subsection, we have shown the asymptotic formula \eqref{eq:eign2:1} for simple eigenvalues. In this subsection, instead, we follow the lines of the arguments of \cite{AbLeMu22} and we consider the case where the $N$-th eigenvalue $\lambda_N(\Omega)$ for the Dirichlet-Laplacian  is multiple.

So let $\lambda_N(\Omega)$ {be an eigenvalue} of multiplicity $m>1$ of
{the Dirichlet-Laplacian} in $\Omega$  and let $E(\lambda_N(\Omega))$ be the associated eigenspace. In particular, we have
\[
\lambda_{N-1}(\Omega)<\lambda_N(\Omega)=\lambda_{N+j}(\Omega)<\lambda_{N+m}(\Omega) \qquad \forall j=0,\dots,m-1\, .
\]

By \cite[Appendix A]{AbLeMu22}, we have the following result on the decomposition of $E(\lambda_N(\Omega))$.

\begin{prop}\label{prop:DecompES} There exists a decomposition of $E(\lambda_N(\Omega))$ into a sum of orthogonal subspaces
\[E(\lambda_N(\Omega))=E_1\oplus\dots\oplus E_p\] 
and an associated finite decreasing sequence of integers
\[k_1>\dots>k_p\ge0\]
such that, for all $1\le j \le p$, a function in $E_j\setminus\{0\}$ has the order of vanishing $k_j$ at $0$. In addition, such a decomposition is unique. We call it the \emph{order decomposition} of $E(\lambda_N(\Omega))$. 
\end{prop}

{We proceed as in \cite{AbLeMu22} and we introduce some notation.  So let  {$u,v \in E(\lambda_N(\Omega))\setminus \{0\}$. Then $u$ and $v$ are admissible functions in the sense of Definition \ref{d:function} and are not identically zero in a neighborhood of $0$. Then $u$ has a finite order of vanishing $\kappa(u)$ at $0$, with a corresponding principal part $u_\#$, meaning that 
\[
\begin{split}
&\
 D^\gamma u(0)=0 \qquad \forall |\gamma| <\kappa(u)\\
&\text{and that $D^\beta u(0)\neq 0$ for some $\beta \in \mathbb{N}^d$ with $|\beta|=\kappa(u)$}\, ,
\end{split}
\]
and
\begin{equation}\label{eq:uNsharp}
u_{\#}(t)\equiv\sum_{\substack{\beta\in \mathbb{N}^d\\ |\beta|=\kappa(u)}}\frac{D^\beta u(0)}{\beta !}t^\beta \qquad\qquad\quad  \forall t\in  \mathbb{R}^d\, .
\end{equation}
Consistent with Equation \eqref{eq:uNbf_0},} we denote by $\mathsf{u}$ the unique solution in $C^{1,\alpha}_{\mathrm{loc}}(\mathbb{R}^d\setminus \omega)$ of
\[
\left\{
\begin{array}{ll}
\Delta \mathsf{u}=0&\text{in }\mathbb{R}^d\setminus\overline{\omega}\,,\\
\mathsf{u}(t)=u_{\#}(t)&\text{for all }t\in\partial\omega\,,\\
 \lim_{t\to \infty}\mathsf{u}(t)=0\,.
\end{array} 
\right.
\]
 
We {write the order
$\kappa(v)$ and the functions $v_{\#}$,  $\mathsf{v}$ also for the function $v$.}  As in \cite{AbLeMu22}, we define
\[
\begin{split}
	\mathcal Q(u,v)\equiv&\int_{\mathbb{R}^d \setminus \overline{\omega}}\nabla \mathsf{u} \cdot \nabla \mathsf{v}\, dt+\int_{\omega} \nabla u_{\#}\cdot \nabla u_{\#} \, dt\, .
\end{split}
\]
{If {instead} $u$ or $v$ is identically zero, we set $\mathcal Q(u,v)=0$.} We  observe that $\mathcal Q$ is \emph{not} a bilinear form, but the restriction of $\mathcal Q$ to suitable subspaces of $E(\lambda_N(\Omega))$ defines bilinear forms (cf.~Definition \ref{def:Qj} below).}

By Remark \ref{rem:capepsfirst} and Theorem \ref{thm:cepsmseries}, we can deduce the following result where we link the asymptotic behavior of  $\mathrm{Cap}_\Omega(\eps\overline\omega, u,v)$ with $\mathcal Q(u,v)$.

\begin{cor}\label{cor:asymptCap} Let us fix {$u,v \in E(\lambda_N(\Omega))\setminus \{0\}$}. Then,
	\begin{equation*}
		\mathrm{Cap}_\Omega(\eps\overline\omega, u,v)=\eps^{\kappa(u)+\kappa(v)+d-2}\mathcal Q(u,v)+o\left(\eps^{\kappa(u)+\kappa(v)+d-2}\right)\mbox{ as }\eps\to 0^+.
	\end{equation*}
\end{cor}

We now consider again an eigenvalue $\lambda_N(\Omega)$ of multiplicity $m>1$ and the associated eigenspace  $E(\lambda_N(\Omega))$ and we give the following definition (cf.~Proposition \ref{prop:DecompES}).

\begin{defn}\label{def:Qj} For all $1\le j\le p$, we define $\mathcal Q_j$ on $E_j$  by $\mathcal Q_j(u,v)\equiv\mathcal Q(u,v)$. It is a strictly positive (in particular non-degenerate) symmetric bilinear form on $E_j$.
\end{defn}

We can now describe the behavior of the eigenvalues $(\lambda_i(\Omega \setminus (\varepsilon \overline{\omega})))_{N\le i\le N+m-1}$, and more specifically give the principal part of the spectral shift $\lambda_i(\Omega \setminus (\varepsilon \overline{\omega}))-\lambda_N(\Omega)$ for each eigenvalue branch departing from $\lambda_N(\Omega)$. 
\begin{theorem} \label{thm:orderEVs} 
For $1\le j\le p$, we write 
\[m_j\equiv\mbox{dim}(E_j),\]
so that 
\[m=m_1+\dots+m_j+\dots+m_p,\]
and we denote by 
\[0<\mu_{j,1}\le\dots\le\mu_{j,\ell}\le\dots\le\mu_{j,m_j}\]
the eigenvalues of the quadratic form $\mathcal{Q}_j$. Then, for all $1\le j\le p$ and $1\le \ell \le m_j$,
\[
\lambda_{N-1+m_1+\dots+m_{j-1}+\ell}(\Omega \setminus (\varepsilon \overline{\omega}))=\lambda_N(\Omega)+\mu_{j,\ell}\,\varepsilon^{2k_j+(d-2)}+o(\varepsilon^{2k_j+(d-2)})\mbox{ as }\eps\to0^+.
\]
\end{theorem}
{

The proof of Theorem \ref{thm:orderEVs} is identical with that of  \cite[Th. 1.17]{AbLeMu22}.  To see how the above results imply Theorem \ref{t:2}, we fix, for each $1\le j\le p$, an orthonormal basis of $E_j$,
\begin{equation*}
	v_{j,1},\dots,v_{j,\ell},\dots,v_{j,m_j}
\end{equation*}  
such that, for all $1\le \ell,\ell'\le m_j$,
\begin{equation*}
	\mathcal Q_j(v_{j,\ell},v_{j,\ell'})=\delta_{\ell\ell'}\mu_{j,\ell}. 
\end{equation*}
From Theorem \ref{thm:orderEVs} and Corollary \ref{cor:asymptCap}, it then follows that, for all $1\le j\le p$ and $1\le \ell \le m_j$, 
\begin{equation*}
	\lambda_{N-1+m_1+\dots+m_{j-1}+\ell}(\Omega \setminus (\varepsilon \overline{\omega}))=\lambda_N(\Omega)+\mathrm{Cap}_\Omega(\eps\overline\omega,v_{j,\ell})+o(\mathrm{Cap}_\Omega(\eps\overline\omega,v_{j,\ell}))\mbox{ as }\eps\to0^+.
\end{equation*}
By relabeling the $v_{j,\ell}$ with $1\le i\le m$, in increasing order, first of $j$, then of $\ell$, we obtain an orthonormal basis of $E(\lambda_N(\Omega))$,
\begin{equation*}
	v_1,\dots,v_i\dots,v_m,
\end{equation*} 	
which has the properties of Theorem \ref{t:2}.
}

Then, to illustrate our result, as in {\cite[Cor.~1.18]{AbLeMu22}, we 
deduce} the following corollary in a specific situation .

\begin{cor} \label{cor:doubleEV} Let us assume that $\lambda_N(\Omega)$ has multiplicity $2$ ({\it i.e.}, $m=2$). Then one of the following alternatives holds.
\begin{enumerate}
\item  There exist two normalized eigenfunctions $u_1,u_2\in E(\lambda_N(\Omega))\setminus\{0\}$, with respective order of vanishing $k_1,k_2$ such that $k_1>k_2$. In that case, 

\begin{align*}
	\lambda_{N}(\Omega \setminus (\varepsilon \overline{\omega}))= &\lambda_{N}(\Omega)+\mathcal Q(u_1,u_1)\eps^{2k_1+(d-2)}+o(\eps^{2k_1+(d-2)});\\
	\lambda_{N+1}(\Omega \setminus (\varepsilon \overline{\omega}))= &\lambda_{N}(\Omega)+\mathcal Q(u_2,u_2)\eps^{2k_2+(d-2)}+o(\eps^{2k_2+(d-2)}).
\end{align*}

\item All eigenfunctions in $E(\lambda_N(\Omega))$ have the same order of vanishing, which we denote by $k$. Let us note that necessarily $k\ge1$. In that case, let us choose eigenfunctions $u_1,u_2$ forming an orthonormal basis of $E(\lambda_N(\Omega))$ and let us denote by $0<\mu_1\le\mu_2$ the eigenvalues of the symmetric and positive {definite} matrix
\begin{equation*}
		\left(\begin{array}{cc}
				\mathcal Q(u_1,u_1)& \mathcal Q(u_1,u_2)\\
				\mathcal Q(u_1,u_2)& \mathcal Q(u_2,u_2)
			\end{array}
		\right).
\end{equation*}
Then
\begin{align*}
	\lambda_N(\Omega \setminus (\varepsilon \overline{\omega}))= & \lambda_{N}(\Omega)+\mu_1\eps^{2k+(d-2)}+o\left(\eps^{2k+(d-2)}\right);\\
	\lambda_{N+1}(\Omega \setminus (\varepsilon \overline{\omega}))= & \lambda_{N}(\Omega)+\mu_2\eps^{2k+(d-2)}+o\left(\eps^{2k+(d-2)}\right).
\end{align*}
\end{enumerate}
\end{cor}

}

\appendix

\section{Blow-up analysis for the $u$-capacity}

Following the outline proposed in \cite{FeNoOgCalcVar2021}, we perform a blow up analysis for $\mathrm{Cap}_{\Omega} (\eps\overline{\omega}, u)$, with $\Omega$, $\omega$ admissible domains and $u$ an admissible function. {Without loss of generality, we  assume that $u$ is not identically zero {in a neighborhood of $0$}.} In the final remark of this appendix we connect this with the results established in {Section \ref{s:vanish}}.

Firstly we introduce the following notation. {Let $K\subseteq \mathbb R^d$ be a compact set. For $d\geq3$, the Beppo--Levi spaces $\mathcal D^{1,2}(\R^d)$ and $\mathcal D^{1,2}(\R^d\setminus K)$ are defined as {the completion, with respect to the $L^2$-norm of the gradient, of $C^\infty_c (\mathbb{R}^d)$ and $C^\infty_c (\mathbb{R}^d\setminus K)$,  respectively. }
We recall as well that thanks to the well-known Hardy Inequality, 
\begin{equation}\label{eq:hardy}
\left(\dfrac{d-2}{2}\right)^2\int_{\mathbb R^d}\frac{{v}^2}{|x|^2} \leq  \int_{\mathbb R^d} |\nabla {v}|^2 \quad \text{(\emph{Hardy Inequality})},
\end{equation}
for all ${v}\in C^\infty_c (\mathbb{R}^d)$, the space $\mathcal D^{1,2}(\mathbb R^d)$ can be characterized in the following way
\begin{equation}\label{eq:charBL}
 \mathcal D^{1,2}(\mathbb R^d) = \left\{ {v}\in L^1_{{\mathrm{loc}}}(\mathbb R^d): \nabla {v}\in L^2(\mathbb R^d) \text{ and }\frac{{v}}{|x|}\in L^2(\mathbb R^d)  \right\}.
\end{equation}

For any $f\in H^1_{{\mathrm{loc}}}(\mathbb R^d)$ we introduce the quantity
\[
 \mathrm{Cap}_{\R^d}(K,f):= \inf\left\{ \int_{\mathbb R^d}|\nabla {v}|^2:\ {v}\in\mathcal D^{1,2}(\mathbb R^d),\ {v}-\eta_K f \in \mathcal D^{1,2}(\mathbb R^d\setminus K) \right\},
\]
where $\eta_K\in C^\infty_c(\mathbb R^d)$ such that $\eta_K=1$ in a neighborhood of $K$. In particular, when $f=1$ we {have
\[\mathrm{Cap}_{\R^d}(K,1)=\mathrm{Cap}_{\R^d}(K)\] 
(for the definition of $\mathrm{Cap}_{\R^d}(K)$ see \eqref{eq:Newcap}).}  Let now $u$ be a fixed admissible function, in the sense of Definition \ref{d:function}, not identically zero in a neighborhood of $0$.  We denote by $\overline k$ its order of vanishing $\kappa(u)$ at $0$ and we define
 \begin{equation*}
  \tilde{u}_\eps(t)\equiv
 \begin{cases}
 \dfrac{u(\eps t)}{\eps^{\overline{k}}} \qquad &\forall t \in \frac1\eps\Omega\,,\\
 0 &\forall t\in \mathbb R^d\setminus \frac1\eps\Omega\, .
 \end{cases}
\end{equation*}
Since $u$ is analytic in a neighborhood of $0$, $\tilde{u}_\eps$ converges to $u_\#$, the  principal part of $u$ defined  
 by Equation \eqref{eq:uNsharp}, uniformly in every compact subset of $\R^d$, and similarly for the derivatives of $\tilde u_\eps$ to any order. This implies in turn that for any $R>0$, as $\eps\to0$, 
\begin{align}
    &\dfrac{1}{\eps^{d+2\overline k}} \int_{\Omega \cap B_{R\eps}} {u}^2 \to \int_{B_R}{u_\#}^2,\label{eq:l2norm}\\
    &\dfrac{1}{\eps^{d-2+2\overline k}}\int_{\Omega \cap B_{R\eps}} |\nabla u|^2 \to \int_{B_R} |\nabla u_\#|^2.\label{eq:l2gradnorm}
\end{align}
{Here, if $\tilde{R}$ is a positive real number, the symbol $B_{\tilde{R}}$ denotes the open ball in $\mathbb{R}^d$ of radius $\tilde{R}$ and center $0$.} Moreover, if $K$ is a compact set with $\mathrm{Cap}_{\R^d}(K)>0$, we have that, for any $R>0$ such that $K\subset B_R$, there exists $C_{P}>0$ such that, {for any $v$ in the closure of $C^\infty_c(\overline{B_R}\setminus K)$ in $H^1(B_R)$,}
{\begin{align}\label{eq:poincare}
\int_{B_R} {v}^2 \leq C_P \int_{B_R} |\nabla {v}|^2 \quad  \quad \text{(\emph{Poincar\'{e} Inequality}).} 
\end{align}}

{ For a proof of the basic inequalities \eqref{eq:hardy} and \eqref{eq:poincare} {in a similar context} we refer the reader to \cite[Lemma 6.5]{FeNoOgCalcVar2021}} and \cite[Lemma 6.7]{FeNoOgCalcVar2021}, respectively.

In the following lemma, we deduce the vanishing rate of $\mathrm{Cap}_{\Omega}(\eps\overline \omega, u)$ as $\eps \to 0$.

\begin{lem}\label{l:Ogrande}
For $\eps\to0$,
\[
\mathrm{Cap}_{\Omega}(\eps\overline \omega, u) =O(\eps^{d-2+2\overline k}).
\]
\end{lem}
\begin{proof}
If $\eps$ is sufficiently small there exists $R>0$ such that $ \eps \overline \omega \subset B_{R\eps}\subset B_{2R\eps}\subset \Omega$. Let $\eta_\eps \in C^\infty_c(\Omega)$ such that $0\leq \eta_\eps \leq 1$, $\eta_\eps\equiv 1 $ in $ B_{R\eps}$, $\eta_\eps\equiv 0 $ in $\Omega\setminus B_{2R\eps}$ and $|\nabla \eta_\eps|\leq \frac{4}{\eps R}$ in $\Omega$. Then $\eta_\eps u \in H^1_0(\Omega)$ and $\eta_\eps u - u \in H^1_0(\Omega \setminus \eps\overline{\omega})$. Therefore
\begin{equation*}
  \mathrm{Cap}_\Omega (\eps\overline{\omega}, {u}) \leq \int_\Omega |\nabla (\eta_\eps u)|^2 = \int_\Omega |\eta_\eps\nabla u + u \nabla \eta_\eps|^2 \leq 2 \int_\Omega {\eta_\eps}^2 |\nabla u|^2 + 2 \int_\Omega {u}^2 |\nabla \eta_\eps|^2 = O(\eps^{d-2+2\overline{k}}),
\end{equation*}
thanks to \eqref{eq:l2norm} and \eqref{eq:l2gradnorm} together with the properties of $\eta_\eps$.
\end{proof}

We now define a suitable rescaling of the $u$-capacitary potential, that is 
\begin{equation*}
  \Tilde{V_\eps}(t)\equiv
 \begin{cases}
 \dfrac{V_{\eps\overline{\omega},u}(\eps t)}{\eps^{\overline{k}}} \qquad &\forall t \in \frac1\eps\Omega\, ,\\
 0 &\forall t\in \mathbb R^d\setminus \frac1\eps\Omega\, ,
 \end{cases}
\end{equation*}
for any positive $\eps$. Let us note that $\Tilde{V_\eps}\in H^1_{{\mathrm{loc}}}(\mathbb R^d)$.
 \begin{lem}
    Let $R>0$. There exists $C>0$ such that 
    \[
    \|\tilde V_\eps\|_{H^1(B_R)} \leq C
    \]
    if $\eps$ is small enough. 
\end{lem}
\begin{proof}
    Let $R_0>0$ such that $B_{R_0}\subset \Omega$. We have
    \begin{align*}
        & \int_\Omega |\nabla V_{\eps\overline{\omega},u}|^2 \geq \int_{B_{R_0}} |\nabla V_{\eps\overline{\omega},u}|^2 = \eps^{d-2+2\overline{k}}\int_{B_{\frac{R_0}{\eps}}} \dfrac{\big|\nabla \big(V_{\eps\overline{\omega},u}(\eps {t})\big)\big|^2}{\eps^{2\overline{k}}} =\eps^{d-2+2\overline{k}}\int_{B_{\frac{R_0}{\eps}}} |\nabla \tilde V_{\eps}|^2.
    \end{align*}
    From Lemma \ref{l:Ogrande} and the 
    {previous inequality,} we deduce that 
    \[
    \int_{B_{\frac{R_0}{\eps}}} |\nabla \tilde V_{\eps}|^2 = O(1)
    \]
    as $\eps\to0$ and {that} there exists $C_1>0$ such that for any $R>0$ 
    \begin{equation}\label{eq:gradbounded}
    \int_{B_{R}} |\nabla \tilde V_{\eps}|^2 \leq C_1 \quad \text{for any }\eps\in{\bigg]}0,\frac{R_0}{R}{\bigg[}.    
    \end{equation}
    
    On the other hand, 
    \begin{align*}
        \int_{B_R} |\tilde V_\eps|^2 &= \int_{B_R} |\tilde V_\eps - \tilde u_\eps + \tilde u_\eps|^2 \leq 2 \int_{B_R} |\tilde V_\eps - \tilde u_\eps|^2 + 2 \int_{B_R} |\tilde u_\eps|^2 \\
        &\leq 2 C_P\int_{B_R} | \nabla (\tilde V_\eps - \tilde u_\eps)|^2 + 2 \int_{B_R} |\tilde u_\eps|^2 \leq 4 C_P \left( \int_{B_R} | \nabla \tilde V_\eps|^2 + \int_{B_R} | \nabla \tilde u_\eps|^2 \right)+ 2 \int_{B_R} |\tilde u_\eps|^2 
    \end{align*}
    thanks to \eqref{eq:poincare} applied on $\tilde V_\eps - \tilde u_\eps$. Every term in the last sum can be proved to be bounded following the preceding argument and taking into account \eqref{eq:l2norm} and \eqref{eq:l2gradnorm}, so that there exists { $\bar\eps\in ]0,1[$ and} $C_2>0$ such that
    \begin{equation}\label{eq:bounded}
    \int_{B_{R}} | \tilde V_{\eps}|^2 \leq C_2 \quad \text{for any }\eps\in {\Big]}0,{\min\left\{\bar\eps,\tfrac{R_0}{R}\right\}}{\Big[}. 
    \end{equation}
    Equations \eqref{eq:gradbounded} and \eqref{eq:bounded} conclude the proof.
\end{proof}

From this last lemma and a diagonal process over a sequence of $R_n\to +\infty$ we deduce that there exist $W\in H^1_{{\mathrm{loc}}}(\mathbb R^d)$ and a subsequence (still denoted with $\eps$) such that 
\[
\tilde V_\eps \to W \quad \text{as }\eps\to0 
\]
weakly in $H^1(B_R)$, strongly in $L^2(B_R)$ for any $R>0$ and almost everywhere in $\R^d$.

By a change of variables we have
\begin{equation}\label{eq:gradboundedtilde}
 \int_{\mathbb R^d} |\nabla \Tilde V_\eps|^2 = \dfrac1{\eps^{d-2+2\overline k}}\int_{\Omega}|\nabla V_{\eps\overline \omega,u}|^2 =  \dfrac1{\eps^{d-2+2\overline k}} \mathrm{Cap}_{\Omega}(\eps\overline \omega, u) = O(1)
\end{equation}
as $\eps\to0$ thanks to Lemma \ref{l:Ogrande}. Via Hardy Inequality \eqref{eq:hardy} we obtain also
\[
 \int_{\mathbb R^d} \dfrac{| \Tilde V_\eps|^2}{|x|^2}\leq \left( \dfrac{2}{d-2} \right)^2 \int_{\mathbb R^d} |\nabla \Tilde V_\eps|^2 \leq C
\]
uniformly with respect to $\eps$. Via Fatou's {lemma,} we have
\begin{equation}\label{eq:Wfatou}
 \int_{\mathbb R^d} \dfrac{| W|^2}{|x|^2} = \int_{\mathbb R^d} \liminf_{\eps\to0}\dfrac{| \Tilde V_\eps|^2}{|x|^2} \leq \liminf_{\eps\to0}\int_{\mathbb R^d} \dfrac{| \Tilde V_\eps|^2}{|x|^2} \leq \Tilde C
\end{equation}
uniformly with respect to $\eps$ small enough. Moreover, for any $R>0$ we have 
\[
 \int_{B_R}|\nabla W|^2 \leq \liminf_{\eps\to0} \int_{B_R}|\nabla \Tilde V_\eps|^2 \leq \bar C
\]
where the last inequality follows from \eqref{eq:gradboundedtilde}
(the constant $\bar C$ does not depend on $R$). 
We then deduce that 
\begin{equation}\label{eq:gradW}
 \int_{\mathbb R^d}|\nabla W|^2 \leq  \bar C\,.
\end{equation}
From \eqref{eq:Wfatou} and \eqref{eq:gradW} we deduce that $W\in \mathcal D^{1,2}(\R^d)$ using the characterization \eqref{eq:charBL}.

{On the other hand, $V_{\eps \overline{\omega},u}$ weakly solves the problem 
\begin{equation}\label{eq:Vepsomega}
\begin{cases}
-\Delta V_{\eps \overline{\omega},u}=0 &\text{in }\Omega\setminus\eps\overline{\omega},\\
V_{\eps \overline{\omega},u} = u &\text{on }\eps\overline{\omega},\\
V_{\eps \overline{\omega},u}=0 &\text{on }\partial\Omega,
\end{cases}
\end{equation}
and in particular
\[
 \int_{\frac1\eps\Omega\setminus \overline{\omega}} \nabla \tilde V_\eps \cdot \nabla \varphi { = 0}\quad \text{for all }\varphi \in H^1_0(\tfrac1\eps\Omega\setminus \overline{\omega}).
\]
Taking the limit for $\eps\to 0 $ in the previous equation we obtain that for any $R>0$ 
\[
 \int_{B_R \setminus \overline{\omega}} \nabla W\cdot \nabla \varphi =0 \quad  \text{for all }\varphi \in H^1_0(B_R \setminus \overline{\omega})
\]
and therefore
\[
 \int_{\R^d \setminus \overline{\omega}} \nabla W\cdot \nabla \varphi =0 \quad  \text{for all }\varphi \in C^\infty_c(\R^d \setminus \overline{\omega}).
\]
As a consequence, thanks also to the uniform convergence $\tilde u_\eps \to u_\#$ on compact sets, $W$ is a weak solution of the problem  
}
{
\[
\begin{cases}
    -\Delta W=0 &\text{in }\R^d {\setminus \overline{\omega}}\\
    W= u_\# &\text{on }\overline{\omega}\\
    W\in \mathcal D^{1,2}(\R^d)
\end{cases}
\]
By standard variational methods, the problem above for $W$ has a unique weak solution which we denote by $V_{\R^d, u_\#}$, {\it i.e.,}~ $W=V_{\R^d, u_\#}$.}

Let $\eta $ be a cut-off function with compact support such that $\eta\equiv 1$ on $\overline\omega$. Testing the previous equation with $W-\eta u_\# \in \mathcal D^{1,2}(\mathbb R^d\setminus \overline \omega)$ we obtain
\begin{equation}\label{eq:W33}
 \int_{\mathbb R^d\setminus \overline \omega}|\nabla W|^2 = \int_{\mathbb R^d\setminus \overline \omega} \nabla W\cdot \nabla (\eta u_\#).
\end{equation}

Letting $\eta_\eps(x){\equiv}\eta(\tfrac{x}{\eps})$, we can {test equation \eqref{eq:Vepsomega} with} $\varphi = V_{\eps \overline{\omega},u} - \eta_\eps u$, obtaining
\[
 \int_{\Omega\setminus \eps\overline \omega} |\nabla V_{\eps \overline{\omega},u}|^2 = \int_{\Omega\setminus \eps\overline \omega} \nabla V_{\eps \overline{\omega},u}\cdot \nabla (\eta_\eps u). 
\]
By a change of variables we obtain
\begin{equation}\label{eq:external}
 \dfrac{1}{\eps^{d-2+2\overline k}}\int_{\Omega\setminus \eps\overline \omega} |\nabla V_{\eps \overline{\omega},u}|^2 = \int_{\tfrac1\eps\Omega\setminus \overline \omega} \nabla \Tilde V_{\eps}\cdot \nabla (\eta \tilde u_\eps) \to \int_{\mathbb R^d\setminus \overline \omega} \nabla W\cdot \nabla (\eta u_\#) 
\end{equation}
as $\eps\to0$ thanks to the weak convergence $\Tilde V_{\eps}\rightharpoonup W$ in $H^1(B_R)$ for any $R>0$ and the uniform convergence $\tilde u_\eps \to u_\#$ on compact sets. Moreover, the very same convergence implies
\begin{equation}\label{eq:internal}
 \dfrac{1}{\eps^{d-2+2\overline k}}\int_{\eps\overline \omega} |\nabla u|^2 =
 \int_{\overline \omega} |\nabla \tilde u_\eps|^2 \to \int_{\overline \omega} |\nabla u_\# |^2 \quad \text{as }\eps\to0.
\end{equation}

In view of \eqref{eq:internal}, \eqref{eq:external} and \eqref{eq:W33} we deduce that  

\begin{equation}\label{eq:capRd}
\eps^{-(d-2+2\overline{k})}\mathrm{Cap}_\Omega (\eps\overline{\omega}, u) \to \mathrm{Cap}_{\R^d} (\overline{\omega}, u_\#) \quad \text{as }\eps\to 
 0,
\end{equation}
independently from the subsequence.  

{Let us note that,} when $u$ is the eigenfunction $u_N$, $V_{\R^d, u_{N,\#,\overline{k}}|\overline{\omega}}=u_{N,\#,\overline{k}|\overline{\omega}}$ and $V_{\R^d, u_{N,\#,\overline{k}}|\mathbb{R}^d\setminus \omega}=\mathsf{u}_{N,\overline{k}}$. {Equation \eqref{eq:capRd} can therefore be written
\begin{equation*}
\mathrm{Cap}_\Omega (\eps\overline{\omega}, u_N) = \mathrm{Cap}_{\R^d} (\overline{\omega}, u_{N,\#,\overline{k}})\eps^{2\overline{k}+d-2}+o\left(\eps^{2\overline{k}+d-2}\right) \quad \text{as }\eps\to 
 0.
\end{equation*}
}

\begin{rem}
We recall that in our context $\omega$ is an open regular bounded set in $\R^d$ and $u_\#$ is a harmonic homogeneous polynomial of degree $\overline{k}$. As a consequence, the quantity $\mathrm{Cap}_{\R^d} (\overline{\omega}, u_\#)$ is strictly positive and \eqref{eq:capRd} is in fact a sharp asymptotics.

Moreover, from {Remark} \ref{rem:cepsmseries} we deduce that 
{\[
\Bigg(\int_{\mathbb{R}^d \setminus \overline{\omega}}|\nabla \mathsf{u}_{N,\overline{k}}|^2\, dt+\int_{\omega} |\nabla u_{N,\#,\overline{k}}|^2 \, dt\Bigg) = \mathrm{Cap}_{\R^d} (\overline{\omega}, u_{N,\#,\overline{k}}) >0.
\]}
\end{rem}

\begin{rem}
 The blow-up analysis presented above can be also performed for the $(u,v)$-capacity provided suitable modifications.
\end{rem}

\subsection*{Acknowledgment}

L.A. has been supported by MUR grant Dipartimento di Eccellenza 2023-2027.  C.L. and P.M. are members of the ``Gruppo Nazionale per l'Analisi Matematica, la Probabilit\`a e le loro Applicazioni'' (GNAMPA) of the ``Istituto Nazionale di Alta Matematica'' (INdAM) {and acknowledge the support of the ``INdAM GNAMPA Project'' codice CUP\_E53C22001930001  ``Operatori differenziali e integrali in geometria spettrale''.} P.M. {also} acknowledges the support from EU through the H2020-MSCA-RISE-2020 project EffectFact, 
Grant agreement ID: 101008140, and  the support of  the SPIN Project  ``DOMain perturbation problems and INteractions Of scales - DOMINO''  of the Ca' Foscari University of Venice.


\end{document}